\documentclass{article}
\usepackage{amsmath,amssymb,amsthm}
\usepackage{authblk}
\usepackage{eucal}
\usepackage[latin1]{inputenc}

\usepackage{subfigure}
\usepackage{pgf,tikz}
\usepackage{mathrsfs}
\usetikzlibrary{arrows}
\definecolor{qqqqff}{rgb}{0.,0.,1.}

\newcommand\RR{\mathbb{R}}

\renewcommand{\H}{{\mathcal H}}
\renewcommand{\epsilon}{\varepsilon}
\renewcommand{\subset}{\subseteq}
\renewcommand{\supset}{\supseteq}

\newcommand{\M}{\mathcal M}

\newtheorem{theorem}{Theorem}[section]
\newtheorem{proposition}[theorem]{Proposition}
\newtheorem{conjecture}[theorem]{Conjecture}

\newtheorem{lemma}[theorem]{Lemma}
\newtheorem{corollary}[theorem]{Corollary}

\theoremstyle{definition}
\newtheorem{definition}[theorem]{Definition}

\theoremstyle{remark}

\newcommand{\defeq}{:=}

\newcommand{\E}{\boldsymbol E}
\renewcommand{\M}{\mathcal M}

\title{Minimal clusters of four planar regions with the same area}
\author[1]{E. Paolini\thanks{The authors are members of the Gruppo Nazionale per
  l'Analisi Matematica, la Probabilità e le loro Applicazioni (GNAMPA)
of the Istituto Nazionale di Alta Matematica (\mbox{INdAM}).}}
\author[2]{A. Tamagnini}
\affil[1]{Università di Pisa}
\affil[2]{Università degli Studi di Firenze}

\begin{document}
\maketitle
\begin{abstract}
  We prove that the optimal way to enclose and separate four planar regions
  with equal area using the less possible perimeter
  requires all regions to be connected. Moreover, the topology of such
  optimal clusters is uniquely determined.
\end{abstract}

\section{Introduction}\label{sec:introduction}

We consider the problem of enclosing and separating $N$ regions of
$\RR^2$ with prescribed area and with the minimal possible interface length.

The case $N=1$ corresponds to the celebrated isoperimetric
problem whose solution, the circle, was known since antiquity.

For $N\ge 1$ first existence and partial regularity in $\RR^n$ was
given by Almgren~\cite{algrem} while Taylor~\cite{JET1} describes the
singularities for minimizers in $\RR^3$. Existence and regularity
of minimizers in $\RR^2$ was proved
by Morgan~\cite{M} (see also \cite{FM}): the regions of a
minimizer in $\RR^2$ are delimited by a finite number of circular arcs
which meet in triples at their end-points  (see
Theorem~\ref{thm:existence_regularity}).

Foisy et al.~\cite{FABHZ} proved that for $N=2$ in $\RR^2$ the two
regions of any minimizer are delimited by three circular arcs joining
in two points (standard double bubble) and are uniquely determined by
their enclosed areas. Wichiramala~\cite{W} proved that for $N=3$ in $\RR^2$ the
three regions of any minimizer are delimited by six circular arcs
joining in four points. Such configuration (standard triple bubble) is
uniquely determined by the given enclosed areas, as shown by
Montesinos~\cite{MM}.
The case $N=4$ has been considered in \cite{KW} where some partial
information on minimal clusters is obtained.

The minimization problem can be stated also for $N=\infty$
regions with equal areas
(the \emph{honeycomb conjecture}, see~\cite{M2}):
Hales~\cite{H} proved that the hexagonal grid is indeed the solution.

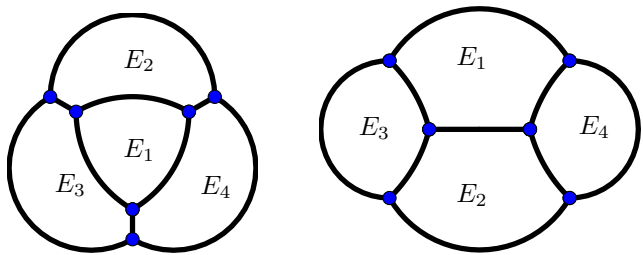
\begin{figure}
\centering
\begin{tikzpicture}[line cap=round,line join=round,>=triangle 45,x=0.5cm,y=0.5cm]
\clip(1.1752267206344569,-0.7410017593858914) rectangle (7.827913116955445,5.618356312582319);
\draw [shift={(4.5,0.401923788646684)},line width=2.pt]  plot[domain=1.047197551196598:2.0943951023931957,variable=\t]({1.*3.*cos(\t r)+0.*3.*sin(\t r)},{0.*3.*cos(\t r)+1.*3.*sin(\t r)});
\draw [shift={(6.,3.)},line width=2.pt]  plot[domain=3.141592653589793:4.188790204786391,variable=\t]({1.*3.*cos(\t r)+0.*3.*sin(\t r)},{0.*3.*cos(\t r)+1.*3.*sin(\t r)});
\draw [shift={(3.,3.)},line width=2.pt]  plot[domain=-1.0471975511965974:0.,variable=\t]({1.*3.*cos(\t r)+0.*3.*sin(\t r)},{0.*3.*cos(\t r)+1.*3.*sin(\t r)});
\draw [line width=2.pt] (3.,3.)-- (2.312320508075691,3.3970319397786835);
\draw [line width=2.pt] (6.,3.)-- (6.687679491924309,3.397031939778683);
\draw [line width=2.pt] (4.5,0.401923788646684)-- (4.5,-0.3921400909106856);
\draw [shift={(4.5,3.3970319397786835)},line width=2.pt]  plot[domain=0.:3.141592653589793,variable=\t]({1.*2.1876794919243094*cos(\t r)+0.*2.1876794919243094*sin(\t r)},{0.*2.1876794919243094*cos(\t r)+1.*2.1876794919243094*sin(\t r)});
\draw [shift={(5.593839745962157,1.5024459244339987)},line width=2.pt]  plot[domain=-2.094395102393195:1.0471975511965979,variable=\t]({1.*2.1876794919243077*cos(\t r)+0.*2.1876794919243077*sin(\t r)},{0.*2.1876794919243077*cos(\t r)+1.*2.1876794919243077*sin(\t r)});
\draw [shift={(3.4061602540378475,1.502445924433999)},line width=2.pt]  plot[domain=2.0943951023931966:5.23598775598299,variable=\t]({1.*2.1876794919243094*cos(\t r)+0.*2.1876794919243094*sin(\t r)},{0.*2.1876794919243094*cos(\t r)+1.*2.1876794919243094*sin(\t r)});
\draw (4.0146449003693405,2.5208092074169914) node[anchor=north west] {${{E_1}}$};
\draw (4.0146449003693405,4.902635201161543) node[anchor=north west] {${{E_2}}$};
\draw (2.1842761564080106,1.5821585694881042) node[anchor=north west] {${{E_3}}$};
\draw (6.056210037864671,1.5352260375916598) node[anchor=north west] {${{E_4}}$};
\draw [fill=qqqqff] (3.,3.) circle (2.5pt);
\draw [fill=qqqqff] (6.,3.) circle (2.5pt);
\draw [fill=qqqqff] (4.5,0.401923788646684) circle (2.5pt);
\draw [fill=qqqqff] (2.312320508075691,3.3970319397786835) circle (2.5pt);
\draw [fill=qqqqff] (6.,3.) circle (2.5pt);
\draw [fill=qqqqff] (6.687679491924309,3.397031939778683) circle (2.5pt);
\draw [fill=qqqqff] (4.5,0.401923788646684) circle (2.5pt);
\draw [fill=qqqqff] (4.5,-0.3921400909106856) circle (2.5pt);
\end{tikzpicture}
\qquad
\begin{tikzpicture}[line cap=round,line join=round,>=triangle 45,x=0.5cm,y=0.5cm]
\clip(9.400305521255785,-1.140201276429096) rectangle (17.942201151423856,5.419229480070645);
\draw [shift={(11.265162377631011,2.133974596215561)},line width=2.pt]  plot[domain=1.5707963267948961:4.71238898038469,variable=\t]({1.*1.8270323556782402*cos(\t r)+0.*1.8270323556782402*sin(\t r)},{0.*1.8270323556782402*cos(\t r)+1.*1.8270323556782402*sin(\t r)});
\draw [shift={(16.05483762236899,2.1339745962155607)},line width=2.pt]  plot[domain=-1.5707963267948966:1.5707963267948966,variable=\t]({1.*1.8270323556782397*cos(\t r)+0.*1.8270323556782397*sin(\t r)},{0.*1.8270323556782397*cos(\t r)+1.*1.8270323556782397*sin(\t r)});
\draw [shift={(13.66,2.578346805953621)},line width=2.pt]  plot[domain=0.5235987755982988:2.6179938779914935,variable=\t]({1.*2.7653202918803603*cos(\t r)+0.*2.7653202918803603*sin(\t r)},{0.*2.7653202918803603*cos(\t r)+1.*2.7653202918803603*sin(\t r)});
\draw [shift={(13.66,1.6896023864775023)},line width=2.pt]  plot[domain=0.5235987755982988:2.6179938779914935,variable=\t]({1.*2.7653202918803603*cos(\t r)+0.*2.7653202918803603*sin(\t r)},{0.*2.7653202918803603*cos(\t r)+-1.*2.7653202918803603*sin(\t r)});
\draw (12.80,4.581006077110232) node[anchor=north west] {${{E_1}}$};
\draw (12.80,0.922094397521127) node[anchor=north west] {${{E_2}}$};
\draw (10.20,2.7049822704845456) node[anchor=north west] {${{E_3}}$};
\draw (16.05,2.7448976706255177) node[anchor=north west] {${{E_4}}$};
\draw [line width=2.pt] (12.32,2.1339745962155607)-- (13.66,2.1339745962155616);
\draw [line width=2.pt] (13.66,2.1339745962155616)-- (15.,2.1339745962155607);
\draw [shift={(8.603304837072304,3.061787978575486)},line width=2.pt]  plot[domain=5.480621894061539:6.038551169101036,variable=\t]({1.*3.830751963598829*cos(\t r)+0.*3.830751963598829*sin(\t r)},{0.*3.830751963598829*cos(\t r)+1.*3.830751963598829*sin(\t r)});
\draw [shift={(8.603304837072285,1.2061612138556204)},line width=2.pt]  plot[domain=0.24463413807855308:0.8025634131180475,variable=\t]({1.*3.830751963598853*cos(\t r)+0.*3.830751963598853*sin(\t r)},{0.*3.830751963598853*cos(\t r)+1.*3.830751963598853*sin(\t r)});
\draw [shift={(18.71669516292769,1.206161213855631)},line width=2.pt]  plot[domain=2.339029240471742:2.8969585155112414,variable=\t]({1.*3.8307519635988254*cos(\t r)+0.*3.8307519635988254*sin(\t r)},{0.*3.8307519635988254*cos(\t r)+1.*3.8307519635988254*sin(\t r)});
\draw [shift={(18.71669516292769,3.061787978575485)},line width=2.pt]  plot[domain=3.386226791668343:3.9441560667078424,variable=\t]({1.*3.830751963598824*cos(\t r)+0.*3.830751963598824*sin(\t r)},{0.*3.830751963598824*cos(\t r)+1.*3.830751963598824*sin(\t r)});
\draw [fill=qqqqff] (12.32,2.1339745962155607) circle (2.5pt);
\draw [fill=qqqqff] (15.,2.1339745962155607) circle (2.5pt);
\draw [fill=qqqqff] (11.265162377631013,3.9610069518938014) circle (2.5pt);
\draw [fill=qqqqff] (11.265162377631011,0.3069422405373212) circle (2.5pt);
\draw [fill=qqqqff] (16.05483762236899,3.961006951893801) circle (2.5pt);
\draw [fill=qqqqff] (16.05483762236899,0.30694224053732094) circle (2.5pt);
\draw [fill=qqqqff] (11.26516237763101,3.961006951893801) circle (2.5pt);
\draw [fill=qqqqff] (12.32,2.133974596215563) circle (2.5pt);
\draw [fill=qqqqff] (16.054837622368993,3.9610069518938005) circle (2.5pt);
\draw [fill=qqqqff] (15.,2.133974596215563) circle (2.5pt);
\draw [fill=qqqqff] (16.05483762236899,0.30694224053732094) circle (2.5pt);
\draw [fill=qqqqff] (15.,2.1339745962155607) circle (2.5pt);
\end{tikzpicture}
\caption{The \emph{flower} (left hand side) and \emph{sandwich} (right
  hand side) topologies.}\label{sandwich_flower}
\end{figure}

In obtaining the results with $N=2$ or $N=3$ planar regions, the main difficulty is to prove
that each region of the minimizer is connected. In fact, in general, this is an open
question
(soap bubble conjecture, Conjecture~\ref{conj:soap_bubble}, see Morgan
and Sullivan~\cite{SM}).

To investigate such a conjecture,
in this paper (which originates from the Ph.D. Thesis~\cite{AT} of the
second author)
we consider the case of $N=4$ regions in the plane.
In Theorem~\ref{thm:connected}
we prove that if the four planar regions have equal areas then the
conjecture is true: the minimizing clusters must be connected.
However, in this case, connectedness and stationarity
is not enough to uniquely determine the topology of minimizers.
In fact there are two nontrivial possible topologies:
we call them the \emph{flower} and the
\emph{sandwich} topologies (see Figure~\ref{sandwich_flower}).
We then exclude the flower
topology, to conclude that minimizers have the sandwich
type (Theorem~\ref{thm:sandwich}).

We conjecture that the minimizer with equal areas is symmetric
i.e.: the regions $E_1$ and $E_3$ are
congruent to the regions $E_2$ and $E_4$ respectively
(see Conjecture~\ref{conj:sandwich} for more details).

The problem of dividing the sphere in regions of equal areas has also
been considered. See for example \cite{E} where it is proven that the
minimizer for four equal areas in the sphere is given by a geodesic tetrahedron.

The plan of the paper is as follows. In Section~\ref{sec:notations} we
set up the notation and collect the known results that we need in the
rest of the paper.
In Section~\ref{sec:variations} we present some
tools which apply to general planar clusters.
In particular notice that Proposition~\ref{prop:varII} gives an estimate by
below on the measure of each connected component of a minimal
cluster. This estimate can be used to obtain an upper bound on the
total number of connected components of a cluster as in
Theorem~\ref{thm:lower_estimate}.

In
Section~\ref{sec:1111} we start the analysis of planar clusters with
four equal areas. In particular we find a precise estimate on the
length of the minimizers (Proposition~\ref{prop:k_0}), we prove
that possible components of a disconnected region cannot be too small
(Proposition~\ref{prop:k_2}) and cannot be too big
(Proposition~\ref{prop:k_1}). This estimates enable us to prove that
a minimizer can have at most six connected components
(Proposition~\ref{prop:two_small}).
In Section~\ref{sec:six} we exclude the clusters with six
components. In Section~\ref{sec:five} we exclude the clusters with
five components and obtain the connectedness result
Theorem~\ref{thm:connected}.
In Section~\ref{sec:four} we consider all connected
clusters (four components) and exclude the flower
topology (Proposition~\ref{prop:no_flower}, Theorem~\ref{thm:sandwich}).

\section{Notation and preliminary results}\label{sec:notations}

Let us denote with $\E = (E_1, \dots, E_N)$ an $N$-uple of
measurable subsets of $\mathbb R^2$. We will say that $\E$ is an
\emph{$N$-cluster} if $m(E_i \cap E_j)=0$ for all $i\neq j$
($m(\cdot)$ is the Lebesgue measure).
The \emph{external region} $E_0$ is defined as
\[
   E_0 = \RR^2 \setminus \bigcup_{i=1}^N E_i.
\]
The sets $E_0, E_1, \dots, E_N$ will be called the \emph{regions} of the
cluster $\E$.

We define the \emph{measure} and the $\emph{perimeter}$ of a cluster by:
\[
  \boldsymbol m(\E) \defeq (m(E_1),\dots,m(E_N)), \qquad
  P(\E) \defeq \frac 1 2 \sum_{i=0}^N P(E_i)
\]
where $P(E_i)$ is the \emph{perimeter} of the measurable set
$E_i$. For regular sets $E_i$ one has
$P(E_i) = \H^{1}(\partial E_i)$
which is the length of the
boundary of $E_i$.

Given a measurable set $E$ we say that $C$ with $m(C)>0$ is a \emph{component} of $E$ if
\[
m(E) = m(C) + m(E\setminus C)
\quad\text{and}\quad
P(E) = P(C) + P(E\setminus C)
\]
(i.e.\ the decomposition $E = C \cup (E\setminus C)$ does not add
any boundary). We say that $E$ is \emph{connected} if it has no
component $C$ with $0 < m(C) < m(E)$ ($C=E$ is a
trivial component).
Notice that in our definitions a \emph{component} does not need to be
connected: in general a component can be a union of connected components.
We say that a cluster $\E$ is
\emph{connected} if each region $E_i$, for $i=1,\dots, N$, is connected.
We say that a cluster is \emph{disconnected} if it is not connected
(i.e.\ at least one region is not connected).

A component $C$ of a region $E_i$ of the cluster $\E$ (with $i\neq 0$) is said to be \emph{external} if is adjacent to the external region $E_0$ (formally $P(C\cup E_0) < P(C)+P(E_0)$) otherwise it is said to be \emph{internal}.

Given a vector of positive numbers $\boldsymbol a \in \RR_+^N$, $\boldsymbol a
= (a_1, \dots, a_N)$, $a_i>0$ we will define the family of
\emph{competitors} as the clusters with measure $\boldsymbol a$:
\[
\mathcal C(\boldsymbol a) = \{ \E \colon \boldsymbol m(\E) =
\boldsymbol a\}
\]
among these we will consider the following optimization problem:
\[
  p(\boldsymbol a) = \inf \{P(\E)\colon \E \in \mathcal C(\boldsymbol a)\}
\]
and the corresponding minimizers:
\[
\M(\boldsymbol a) = \{\E\in\mathcal C(\boldsymbol a) \colon P(\E) = p(\boldsymbol a)\}.
\]

We will also consider the \emph{weak} variants of this minimization
problem:
\begin{align*}
\mathcal C^*(\boldsymbol a) &= \{ \E \colon \boldsymbol m(\E) \ge
\boldsymbol a\}\\
  p^*(\boldsymbol a) &= \inf \{P(\E)\colon \E \in \mathcal C^*(\boldsymbol a)\}\\
  \M^*(\boldsymbol a) &= \{\E\in \mathcal C^*(\boldsymbol a) \colon P(\E) = p^*(\boldsymbol a)\}.
\end{align*}
(the comparison between vectors of $\RR^N$ is understood componentwise).

\begin{definition}[regular cluster]\label{def:regular}
  We say that a planar $N$-cluster $\E$ is
  \emph{regular} when:
  \begin{enumerate}
  \item
    each region (including the external region $E_0$) is (up to a
    negligible set) a closed set which is equal to
    the closure of its interior points (and in the following we will assume that the
    Lebesgue representant of the regions $E_i$ is always a closed set);
  \item
    each region, but the external one $E_0$, is bounded;
  \item
    the boundary of the cluster, defined by
    \[
    \partial \E = \bigcup_{k=1}^N \partial E_k
    \]
    is the continuous embedding of a finite planar graph
    (i.e.\ there are a finite number of simple continuous curves which we will call
    \emph{edges} which can only meet in their end-points which we
    will call \emph{vertices} and the \emph{faces} of the graph
    correspond to the connected components of the regions);
  \item
    each \emph{vertex} has order at least three (i.e.\ it coincides
    with at least three end-points of the edges).
  \end{enumerate}

  Notice that the perimeter of a region $E_i$ of a regular cluster
  $\E$ is the
  sum of the length of the edges of $E_i$. Moreover since each edge
  belongs to the boundary of exactly two regions, we have
  \[
  P(\E) = \frac 1 2 \sum_{k=0}^N P(E_k) =
  \sum_{\text{$\sigma$ edge of $\E$}} \ell(\sigma)
  = \mathcal H^1(\partial \E).
  \]
\end{definition}

\begin{definition}[stationary cluster]\label{def:stationary}
  We say that a regular planar cluster $\E = (E_1, \dots, E_N)$ is \emph{stationary} if it
  satisfies the following conditions:
  \begin{enumerate}
  \item
    every edge
    is either a circular
    arc or a straight segment (which,
    in the following, we will identify with an arc of zero curvature);
  \item
    in every \emph{vertex} exactly three arcs meet, defining three
    equal angles of 120 degrees;
  \item
    it is possible to associate a real number $p_i$ (which we will
    call \emph{pressure}) to each region
    $E_i$ of the cluster, so that $p_0 = 0$ and such that any arc between the regions $E_i$ and $E_j$ has curvature
    $\lvert p_i - p_j \rvert$
    (it is a straight segment if $p_i=p_j$) and the region with higher pressure
    is towards the side where the the arc is convex.
  \end{enumerate}

  In particular it follows that the sum of the signed curvatures of
  the three arcs meeting in a vertex is always zero.
\end{definition}

\begin{theorem}[existence and regularity]\label{thm:existence_regularity}\cite{M, FM}
  Given $\boldsymbol a \in \RR_+^N$ the family of clusters $\M(\boldsymbol a)$
  is not empty and every minimal cluster $\E \in
  \M(\boldsymbol a)$ is regular and stationary.
\end{theorem}

\begin{theorem}[existence and regularity, weak case]\label{thm:weak_existence_regularity}
  \cite{M}
  Given $\boldsymbol a \in \RR_+^N$ the family of clusters $\M^*(\boldsymbol a)$
  is not empty  and every minimal cluster $\E \in
  \M^*(\boldsymbol a)$ is regular and stationary.
\end{theorem}
%


Weak minimizers have some additional
properties which makes them a better ambient space for our investigation.

\begin{proposition}[properties of weak minimizers]~\cite{W}\label{prop:weak_minimizers}
  Let $\E \in \M^*(\boldsymbol a)$, $\boldsymbol a \in \RR_+^N$.
  Then:
  \begin{enumerate}
  \item the external region $E_0$ is connected;
  \item all the pressures $p_i$ are nonnegative;
  \item if $m(E_i) > a_i$ then $p_i=0$.
  \end{enumerate}
\end{proposition}
%





\begin{theorem}[pressure formula]~\cite{CHHKLMT}\label{thm:pressure_formula}
  Let $\E \in \M^*(\boldsymbol a)$ with $\boldsymbol a \in \RR_+^N$. Then
   \[
  P(\E) = 2 \sum_{i=1}^N p_i m(E_i).
  \]
\end{theorem}


\begin{lemma}[turning angle]~\cite{W}\label{lem:turning_angle}
  Let $\E\in\M^*(\boldsymbol a)$, $\boldsymbol a\in \RR_+^N$ and let $C$ be a connected
  component of some region $E_i$ of $\E$. Let $n$ be the number of
  edges of $C$ and let $L_j$ be the total length of the edges of $C$
  in common with the region $E_j$ ($L_j=0$ if $C$ and $E_j$ have not edges in common). Then, if $i\neq 0$, it holds
  \[
  \frac{(6-n)\pi}{3} = \sum_{j=0}^N (p_i - p_j) L_j
  \]
  where $p_j$ is the pressure of the region $E_j$.
  For $i=0$ we have instead
  \[
  \frac{(6+n)\pi}{3} = \sum_{j=1}^N (p_j - p_0) L_j.
  \]
\end{lemma}
%



\begin{proposition}~\cite{B}~\cite{Btre}~\cite{AT}\label{prop:edges}
  Let $\E \in \M^*(\boldsymbol a)$ with $\boldsymbol a \in \RR_+^N$. Let $M$
  be the total number of bounded connected components of the regions of $\E$.
  \begin{enumerate}
    \item Every bounded connected component is simply connected.
    \item Two connected components of $\E$ cannot share more than a single edge.
    \item If $N>2$ then each connected component $C$ of $\E$ has at
      least three edges.
    \item Each connected component of a region with $k$ connected
      components has at most $M+1-k$
      edges, and if it is internal it has at most $M-k$ edges.
    \item The total number of edges is $3(M-1)$ and the total number
      of vertices is $2(M-1)$.
    \item If $M\le 6$ then $\E\in \M(\boldsymbol a)$ (i.e.~$\boldsymbol E$ is a
      strong minimizer).
  \end{enumerate}
\end{proposition}
\begin{theorem}[removal of triangle components]~\cite{W}\label{thm:triangle_remove}
  Let $\E\in \mathcal C(\boldsymbol a)$ be a stationary regular cluster
  and suppose that a connected component $C$ of some region $E_i$ has three
  edges.
  Consider the three arcs which arrive at the three vertices of $C$ but
  are not edges of $C$.
  The circles containing these three arcs meet in a point $P$ inside the component $C$.

  Moreover the cluster $\E'$ obtained from $\E$ by removing the component $C$
  and prolonging the three edges, is itself a stationary regular
  cluster $\E'\in \mathcal C(\boldsymbol a')$ with $a_i' = a_i - m(C)$
  (and the region $E_i$ disappears if $C$ was the only component of $E_i$) and
  $a_j'\ge a_j$ for all $j\neq i$. Also the pressures $p_j'$ of the
  regions of $\E'$ are equal to the pressure $p_j$ of the
  corresponding regions of $\E$ (if $E_i$ disappears because $C$
  was the only component of $E_i$, the regions must be relabeled but
  again the pressures of the corresponding regions remain the same).
\end{theorem}


\begin{theorem}[double bubble monotonicity]~\cite{FABHZ}\label{thm:double_bubble_monotone}
  Given $r_1>0$ and $r_2>0$, up to isometries, there
  exists a unique double bubble $\E$ such that the external
  radii of the two regions $E_1$ and $E_2$ are $r_1$ and $r_2$
  respectively.
  If we increase one radius (say) $r_1$ then the area of the
  corresponding region $E_1$ increases while the area of the other
  region $E_2$ decreases. As a consequence there is a unique double
  bubble with prescribed areas.
\end{theorem}

\begin{conjecture}[soap bubble conjecture]~\cite{SM}\label{conj:soap_bubble}
  For all $\boldsymbol a \in \RR^N_+$ each $\E \in \M(\boldsymbol a)$
  is connected.
\end{conjecture}

The main aim of this paper is to prove that the conjecture holds in
the case $\boldsymbol a = (1,1,1,1)$.

\section{Estimates on general clusters}\label{sec:variations}

\begin{lemma}[isoperimetric inequality for clusters]\label{lem:isop}
  Given $\E\in \mathcal C^*(\boldsymbol a)$ one has
  \[
  P(\E) \ge \sqrt{\pi}\left(\sqrt{\sum_{k=1}^N a_k} + \sum_{k=1}^N\sqrt{a_k}\right).
  \]
\end{lemma}
\begin{proof}
  Given any $\E \in \mathcal C^*(\boldsymbol a)$, by applying the isoperimetric inequality
  \[
  P(E) \ge 2 \sqrt{\pi}\sqrt{\min \{ m(E), m(\boldsymbol R^2\setminus E)\}}
  \]
  one has:
  \begin{align}
    P(\E)
    = \frac 1 2 \sum_{k=0}^N P(E_k)
    \ge \sqrt{\pi}\left(\sqrt{\sum_{k=1}^N m(E_k)} + \sum_{k=1}^N \sqrt{m(E_k)}\right).
  \end{align}
\end{proof}

\begin{proposition}[variation I]\label{prop:varI}
  Let $\E \in \M^*(\boldsymbol a)$ and suppose that $C_i$ is a
  component of the region $E_i$.
  Let $\ell$ be the sum of the
  lengths of the edges of $C_i$ in common with the region $E_k \neq E_i$
  ($k=0$ is also admitted). Then
  \[
     \ell \le 2 \sqrt{\pi}\sqrt{m(C_i)}.
  \]
\end{proposition}

\begin{proof}
  Let $B$ be any ball disjoint from $\E$ with the same area as $C_i$, so
  that $P(B) = 2 \sqrt{\pi}\sqrt{m(C_i)}$.
  Consider the cluster
  $\E'$ obtained by $\E$ by means of the following variations on the
  regions $E_i$ and $E_j$:
  \[
  E_i' = (E_i \setminus C_i ) \cup B,
  \qquad
  E_j' = E_j \cup C_i.
  \]
  Clearly we have $m(E_i') = m(E_i)$ and $m(E_j') > m(E_j)$. Hence
  $\E' \in\mathcal C^*(\boldsymbol a)$. Moreover, since the edges of length
  $\ell$ has been removed and the ball $B$ has been added, by the
  minimality of $\E$ we have:
  \[
  0 \le P(\E') - P(\E) = P(B) - \ell = 2\sqrt{\pi}\sqrt{m(C_i)} - \ell.
  \]
\end{proof}

\begin{proposition}[variation II]\label{prop:varII}
  Let $\E \in \M^*(\boldsymbol a)$ and suppose that $C_i$ is a
  component of the region $E_i$ with $0 < m(C_i) < m(E_i)$.
  Let $\ell_k$ be the sum of the
  lengths of the edges of $C_i$ in common with the region $E_k \neq E_i$
  ($k=0$ is also admitted).
  Then
  \begin{equation}\label{eq:ell_k}
    \ell_k
    \le \frac{m(C_i)}{\lvert 2 a_i-m(C_i)\rvert} P(\E).
  \end{equation}
  Moreover, if we denote by $r\le N$ the number of regions which have
  an edge in common with $C_i$,
  for all $\lambda \ge P(\E)$ one has:
  \begin{align}\label{eq:mC_low}
    m(C_i) &\ge \frac{16 \pi a_i^2}{r^2\lambda^2}\left(1- \frac{16\pi a_i}{r^2\lambda^2}\right).
  \end{align}
\end{proposition}

\begin{proof}
  Let
  \[
  t
  = \sqrt{\frac{m(E_i)}{m(E_i) - m(C_i)}}
  = \sqrt{1 + \frac{m(C_i)}{m(E_i)-m(C_i)}}
  \le 1 + \frac 1 2 \frac {m(C_i)}{m(E_i)-m(C_i)}
  \]
  and consider a new cluster $\E'$ whose regions are defined
  by
  $E'_i = t (E_i   \setminus C_i)$,
  $E'_k = t (E_k \cup C_i)$
  and $E'_j = t E_j$ when $j\not\in\{i,k\}$.
  Simply speaking, the cluster $\E'$ has been obtained from $\E$ by
  giving $C_i$ to $E_k$ and then rescaling of a factor $t>1$.

  Notice that $t$ was defined so that
  \[
  m(E'_i) = t^2 (m(E_i) - m(C_i)) = m(E_i)
  \]
  and clearly every other region does not decrease its measure
  since $t > 1$.
  So  $\E' \in \mathcal C^*(\boldsymbol a)$ is a weak competitor to
  $\boldsymbol  E$.
  On the other hand since in the cluster $\E'$ all edges in
  common between the component $t C_i$
  and the region $t E_k$ have been removed (and these edges have
  a total length of $t \ell_k$) we have
  \[
  P(\E') = t (P(\E) - \ell_k).
  \]
  Since $P(\E) \le P(\E')$ one obtains:
  \begin{align*}
  P(\E) & \le t (P(\E) - \ell_k)
  \le \left(1+\frac{m(C_i)}{2(m(E_i) - m(C_i))}\right)(P(\E) -
  \ell_k) \\
  &= P(\E) + \frac{m(C_i)}{2(m(E_i)-m(C_i))} P(\E)
    - \frac{2m(E_i)-m(C_i)}{2(m(E_i)-m(C_i))} \ell_k
  \end{align*}
  which is equivalent to
  \begin{equation*}
      \ell_k \le \frac{m(C_i)}{2m(E_i)-m(C_i)} P(\E).
  \end{equation*}
  Using $0\le a_i \le m(E_i)$ and $m(C_i)\le m(E_i)$ one can easily check that
  \[
  2 m(E_i)-m(C_i) \ge \lvert 2a_i - m(C_i)\rvert
  \]
  so that \eqref{eq:ell_k} is proven.

  Now if the component $C_i$ has edges in common with at least $r$ other
  regions, there is $k$ such that $\ell_k \ge P(C_i)/r$. By also
  applying the isoperimetric inequality
  $P(C_i) \ge 2 \sqrt \pi\sqrt{m(C_i)}$ we obtain:
  \[
  2 \sqrt{\pi}\sqrt{m(C_i)} \le r \ell_k \le \frac{r m(C_i)}{\lvert 2 a_i - m(C_i)\rvert}P(\E) \le
  \frac{r\lambda m(C_i)}{\lvert 2a_i -m(C_i)\rvert}
  \]
  if $P(E) \le \lambda$ as in the statement of the Theorem being
  proved.
  Whence,
  by squaring and then dividing by $m(C_i)$, we obtain
  \[
  4 \pi \le \frac{r^2\lambda^2 m(C_i)}{(2a_i - m(C_i))^2}
  = \frac{r^2\lambda^2 m(C_i)}{4a_i^2 - 4 a_i m(C_i) + m^2(C_i)}
  \]
  which is equivalent to the
  following quadratic inequality in $m(C_i)$:
  \[
  m^2(C_i) - \left(4 a_i + \frac{r^2\lambda^2}{4\pi}\right)m(C_i) + 4 a_i^2
  \le 0.
  \]
  The corresponding equation has two positive solutions, and $m(C_i)$
  is larger than the smaller of the two. So we obtain:
  \begin{equation}\label{eq:mC_low1}
    \begin{aligned}
    m(C_i) &\ge 2 a_i +\frac{r^2\lambda^2}{8\pi} - \sqrt{\left(2a_i +
      \frac{r^2\lambda^2}{8\pi}\right)^2 - 4 a_i^2} \\
    &= 2 a_i - \frac{r^2\lambda^2}{8\pi}\left(\sqrt{1+\frac{32 \pi a_i}{r^2\lambda^2}}-1\right).
    \end{aligned}
  \end{equation}
  By using the inequality:
  \[
  \sqrt{1+x} \le 1 + \frac{x}{2} - \frac{x^2}{8} + \frac{x^3}{16}
  \]
  with $x=\frac{32 \pi a_i}{r^2\lambda^2}$, after some straightforward
  simplifications, we obtain~\eqref{eq:mC_low}.

\end{proof}

The following result is not used in the rest of the paper, but might
be interesting by itself.

\begin{theorem}\label{thm:lower_estimate}
  Let $\E \in \M^*(\boldsymbol a)$ be an $N$-cluster with $N\ge 3$
  and suppose that $C_i$ is a
  component of the region $E_i$ with $0 < m(C_i) < m(E_i)$ and suppose
  that $r$ is the number of regions which are adjacent to $C_i$.

  Let
  \[
  \left\Vert \boldsymbol a\right\Vert_{\frac 1 2}
  = \left(\sum_{j=1}^N \sqrt{a_j}\right)^2,
  \qquad
  \left\Vert \boldsymbol a\right\Vert_{-1}
  = \left(\sum_{j=1}^N (a_j)^{-1}\right)^{-1}.
  \]

  Then
  \begin{equation}\label{eq:lower_estimate}
    m(C_i)
    \ge \frac{20}{9}\frac{a_i^2}{r^2\left\Vert \boldsymbol
      a\right\Vert_{\frac 1 2}}\geq \frac{20}{9}\frac{a_i^2}{N^2\left\Vert \boldsymbol
      a\right\Vert_{\frac 1 2}}.
  \end{equation}
  In particular, the number $M_i$ of connected components of $E_i$ has
  the following bound
  \[
  M_i \le \frac{9}{20}N^2\frac{\Vert \boldsymbol a\Vert_{\frac 1 2}}{a_i}
  \]
  and hence the total number $M$ of connected components of $\E$ is
  bounded by
  \[
  M \le \frac{9}{20}N^2\frac{\Vert \boldsymbol a\Vert_{\frac 1 2}}{\Vert
 \boldsymbol a\Vert_{-1}}.
  \]
\end{theorem}
\begin{proof}
  Consider, as a competitor, a cluster $\E'$ whose regions
  $E'_i$ are disjoint balls with area $a_i$ and let
  \[
  \lambda = P(\E') = 2\sqrt{\pi} \sum_{j=1}^N \sqrt{a_j}
  = 2 \sqrt{\pi}\sqrt{\left\Vert \boldsymbol a\right\Vert_{\frac 1 2}}
  \]
  Since $\E'\in \mathcal C^*(\boldsymbol a)$ we have
  $P(\E)\le \lambda$.
  Notice that $\E \in \M^*(\boldsymbol a)$ implies that $\E \in
  \M^*(\boldsymbol a^*)$ with $\boldsymbol a^* = \boldsymbol m(\E)$,
  we can apply Proposition~\ref{prop:varII} with $\boldsymbol a^*$ in place
  of $\boldsymbol a$ and with $\lambda$ defined as above.
  So \eqref{eq:mC_low} holds with this value of $\lambda$ and
  $\boldsymbol a^*$ in place of $\boldsymbol a$.

  Notice also that $\lambda = P(\E') \ge P(\E) \ge P(E_i) \ge 2
  \sqrt{\pi} \sqrt{m(E_i)} = 2\sqrt{\pi}\sqrt{a^*_i}$.
  Moreover $r\ge 3$
  since, by Proposition~\ref{prop:edges}, we know that for $N\ge 3$ every
  component has at least three edges.  Hence we know that
  \[
  1 - \frac{16\pi a^*i}{r^2\lambda^2}
  \ge 1 - \frac{16\pi a^*_i}{9\cdot 4 \pi a^*_i} = \frac{5}{9}.
  \]
  So \eqref{eq:mC_low} becomes (notice that $r\leq N$)
  \begin{align*}
    m(C_i)
    & \ge\frac{16\pi (a^*_i)^2}{4 \pi r^2 \left\Vert \boldsymbol a\right\Vert_{\frac 1 2}}\cdot \frac 5 9
     = \frac{20}{9} \cdot\frac{(a^*_i)^2}{r^2\left\Vert \boldsymbol a\right\Vert_{\frac 1 2}}\geq
     \frac{20}{9} \cdot\frac{(a^*_i)^2}{N^2\left\Vert \boldsymbol a\right\Vert_{\frac 1 2}}
  \end{align*}
  and, noting that $a^*_i = m(E_i)\ge a_i$, \eqref{eq:lower_estimate} is proved.

  Now suppose that $C_i$ be the component of $E_i$ with smaller area. Then
  $a^*_i = m(E_i) \ge M_i \cdot m(C_i)$ and we have
  \[
  M_i \le \frac{a^*_i}{m(C_i)}
  \le \frac{a^*_i}{\frac{20}{9}\frac{(a^*_i)^2}{N^2\Vert \boldsymbol a\Vert_{\frac 1 2}}}
  = \frac{9}{20}\cdot \frac{N^2 \Vert \boldsymbol a\Vert_{\frac 1 2}}{a^*_i}
  \le \frac{9}{20}\cdot \frac{N^2 \Vert \boldsymbol a\Vert_{\frac 1 2}}{a_i}
  \]
  and summing up for $i=1,\dots, N$ we obtain:
  \[
  M = \sum_{i=1}^N M_i
  \le \frac{9}{20}N^2 \Vert \boldsymbol
  a\Vert_{\frac 1 2} \sum_{i=1}^N \frac{1}{a_i}
  = \frac{9}{20} N^2 \frac{\Vert \boldsymbol a\Vert_{\frac 1 2}}{\Vert
  \boldsymbol a\Vert_{-1}}.
  \]
\end{proof}

\begin{proposition}\label{prop:pressure_perimeter}
  Let $\E\in \M^*(\boldsymbol a)$ and let $C$ be a connected component of
  some region $E_i$. Let $n$ be the number of edges of $C$. Then we
  have the following estimate on the pressure of the region $E_i$:
  \[
  p_i \ge \frac{(6-n)\pi}{3P(C)} + \left(1-\frac{\ell}{P(C)}\right) p_{\mathrm{min}}
  \ge \frac{(6-n)\pi}{3P(C)}
  \]
  where $\ell$ is the length of the external edge of $C$ ($\ell=0$
  if $C$ is internal) and $p_{\mathrm{min}}$ is the lowest pressure of
  the bounded regions
  which are adjacent to $C$.
\end{proposition}
\begin{proof}
  By Lemma~\ref{lem:turning_angle} we have
  \begin{align*}
  \frac{(6-n)\pi}{3}
  &= \sum_j (p_i - p_j) L_j
  = p_i \sum_j L_j - \sum_{j\neq 0} p_j L_j\\
  &\le p_i \sum_j L_j  - p_{\mathrm{min}} \sum_{j\neq 0}  L_j
  = p_i P(C) - p_{\mathrm{min}} (P(C)-\ell)
  \end{align*}
  where the sum in $j$ is extended to the regions $E_j$ which are
  adjacent to $C$. The first estimate of the statement follows.

  To get the second estimate recall that $p_{\mathrm{min}}\ge 0$ in view of Proposition~\ref{prop:weak_minimizers}.
\end{proof}

\begin{proposition}[variation III]\label{prop:pressure_length}
  Let $\E \in \M^*(\boldsymbol a)$ be a cluster and let $B$ and $C$ be two
  different components of the same bounded region $E_i$ of $\E$. Let
  $p_i$ be the pressure of $E_i$.
  Suppose that $B$ is external and let $L$ be the length of the
  external arc of $B$ and $n$ be the number of different regions
  which are adjacent to $C$.
  Then
  \[
  p_i \ge \frac{P(C)}{n\, m(C)} - \frac{2}{L}
  \ge \frac{2 \sqrt{\pi}}{n\sqrt{m(C)}} - \frac{2}{L}.
  \]

\end{proposition}
\begin{proof}
  Suppose $i=1$ and consider all
  the regions which are adjacent to $C$. Suppose that $E_2$ is
  the region
  whose edges in common with $C$ have largest total
  length. Let $\ell$ be such total length in common between $C$ and $E_2$:
  we have that $n\ell \ge P(C)$.

  Let $\gamma$ be the external edge of $B$ and let $v$ and $w$ be its
  vertices. The arc $\gamma$ has
  radius $R = 1/p_1$, length $L$ and spans an angle $\theta = L/R$.
  Given $h>0$ we are going to modify $B$ by
  increasing the radius $R$ up to $R+h$. Just consider the two radii
  in $v$ and $w$: extend them of a length $h$ and join them with a
  parallel arc of radius $R+h$. Let $D$ be the strip
  between these two parallel arcs. We have $m(D) = ((R+h)^2 -
  R^2)\theta/2 = Lh + Lh^2/(2R)\ge Lh$.
  It is easy to see that $D\subset E_0$ (since all the external
  arcs are convex and meet at angles of 120 degrees). Fix $h = m(C) /
  L$ and consider the following variation:
  \[
  E_1' = (E_1 \setminus C) \cup D, \qquad
  E_2' = E_2 \cup C.
  \]
  If we let $\E' = (E_1', E_2', E_3, \dots, E_N)$ we notice that
  $m(E_1')\ge m(E_1)$ (since $m(D)\ge Lh = m(C)$) so
  $\E'\in  \mathcal C^*(\boldsymbol a)$. Moreover, in computing the perimeter of
  $\E'$ the edges in common between $C$ and $E_2$ have been removed so
  we gain $\ell$ while the arc of length $L$ has increased to length
  $2h + L(R+h)/R$ and so we have, by the minimality of $\E$:
  \[
  0
  \le P(\E') - P(\E)
  \le - \ell + 2h + L\frac{R+h}{R} - L
  = m(C)\left(\frac{1}{R} + \frac{2}{L}\right)- \ell
  \]
  To obtain the statement just remember that $1/R = p_1$
  and remember that $\ell \ge P(C)/n$.
\end{proof}


\begin{lemma}\label{lem:double_bubble}
  Let $\E\in \mathcal C(a_1,a_2)$ be a connected stationary cluster (a
  double bubble) with $a_1 \ge a_2$.
  Then the pressures $p_1$, $p_2$ satisfy the
  following relations
  \[
  \frac{k_8}{\sqrt{a_1}} \le p_1 \le p_2 \le \frac{k_8}{\sqrt{a_2}}
  \]
  with
  \[
  k_8 \defeq \sqrt{\frac{2\pi}{3}+\frac{\sqrt 3}{4}},
  \qquad
    1.5897 < k_8 <  1.5898.
  \]
\end{lemma}
\begin{proof}
  By Theorem~\ref{thm:double_bubble_monotone} we know that the
  external radii $r_1,r_2$ and areas $a_1,a_2$ of a double bubble
  are in one-to-one correspondence. Moreover we know that when
  $r_1=r_2$ we have $a_1=a_2$ because the resulting double bubble is
  symmetric and a direct computation gives $a_1 = a_2 = k_8^2 r^2$
  (with $r=r_1=r_2$).
  Hence, by the monotonicity proven in
  Theorem~\ref{thm:double_bubble_monotone}, since we 
  have $a_1\ge a_2$ by assumption, we know that $r_1 \ge r_2$ and
  hence $p_1 \le p_2$ (remember that $p_i = 1/r_i$).
  Hence monotonicity gives also:
  \[
  a_1 \ge k_8^2 r_1^2,\qquad
  a_2 \le k_8^2 r_2^2
  \]
  whence
  \[
    p_1 = \frac{1}{r_1} \ge \frac{k_8}{\sqrt{a_1}},\qquad
    p_2 = \frac{1}{r_2} \le \frac{k_8}{\sqrt{a_2}}.
  \]
\end{proof}

\begin{lemma}[reduction to double-bubble]\label{lemma:double_bubble_pressure}
  Let $\E=(E_1,\dots, E_N)$ be a stationary cluster which is reducible
  to a double
  bubble $(E'_i,E'_j)$ by subsequent removal of triangular components
  where $E'_i\supset E_i$, $E_j'\supset E_j$,
  $E'_i\subset\RR^2\setminus(E_0\cup E_j)$ and $E'_j\subset\RR^2\setminus(E_0\cup E_i)$.
  Let $\boldsymbol a = \boldsymbol m(\E)$ and $a=\sum_{k=1}^N a_k$.

  Then
  \[
  \frac{k_8}{\sqrt{\max\{a-a_i,a-a_j\}}} \le
  \min\{p_i,p_j\}
  \le \max\{p_i,p_j\}
  \le \frac{k_8}{\sqrt{\min\{a_i,a_j\}}}.
  \]
\end{lemma}
\begin{proof}
  By Theorem~\ref{thm:triangle_remove} we know that the pressures
  of the double bubble are equal to the corresponding pressures of the
  cluster $\E$.
  Also notice that, for $k=i,j$ one has $m(E'_k)\ge m(E_k) = a_k$ $(k=i,j)$,
  while $m(E'_i)\leq m(\RR^2\setminus(E_0\cup E_j))=a-a_j$
   and $m(E'_j)\leq m(\RR^2\setminus(E_0\cup E_i))=a-a_i$ so, by
  Theorem~\ref{thm:double_bubble_monotone} we obtain the desired result.
\end{proof}

\begin{lemma}[perimeter of triple bubble]\label{lemma:ptb}
  One has
  \[
  p^*(1,1,1) = 6\sqrt{\frac{\pi}2 + \frac{1}{\sqrt 3}} \ge k_{10}
  \defeq 8.7939
  \]
\end{lemma}
\begin{proof}
From~\cite{W} we know that each $\E \in \M^*(1,1,1) = \M(1,1,1)$ is a standard
triple bubble where each region $E_i$
is composed by the union of an
half-circle and an isosceles triangle with two angles of $30$
degrees.
If $r$ is the radius of the half circles, the
area of each region turns out to be
 \[
 1 = \left(\frac{\pi}{2} +\frac{1}{\sqrt 3}\right) r^2
 \]
 while the perimeter is given by
 \[
 P(\E) = (3\pi + 2\sqrt{3}) r = 6 \left(\frac \pi 2 + \frac{1}{\sqrt
   3}\right) r = 6 \sqrt{\frac \pi 2 + \frac{1}{\sqrt 3}}.
 \]
\end{proof}

\section{Estimates on $\M^*(1,1,1,1)$}\label{sec:1111}

\begin{proposition}[the competitor]\label{prop:k_0}
  We have
  $p^*(1,1,1,1) \le k_0 \defeq 11.1962$.
\end{proposition}
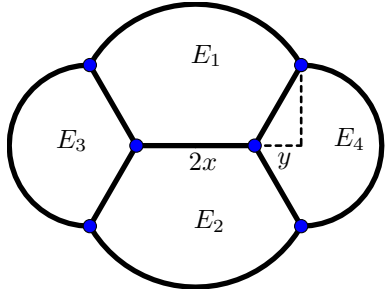
\begin{figure}
  \centering
  \begin{tikzpicture}[line cap=round,line join=round,>=triangle 45,x=0.6cm,y=0.6cm]
\clip(9.518675163482984,-1.0411441577114937) rectangle (17.88544801705282,5.296918320903997);
\draw [shift={(11.348778145295991,2.1339745962155616)},line width=2.pt]  plot[domain=1.5707963267948961:4.71238898038469,variable=\t]({1.*1.786128646222756*cos(\t r)+0.*1.786128646222756*sin(\t r)},{0.*1.786128646222756*cos(\t r)+1.*1.786128646222756*sin(\t r)});
\draw [shift={(16.031221854704015,2.133974596215561)},line width=2.pt]  plot[domain=-1.5707963267948966:1.5707963267948966,variable=\t]({1.*1.7861286462227561*cos(\t r)+0.*1.7861286462227561*sin(\t r)},{0.*1.7861286462227561*cos(\t r)+1.*1.7861286462227561*sin(\t r)});
\draw [line width=2.pt] (12.38,2.1339745962155607)-- (11.348778145295991,3.9201032424383175);
\draw [line width=2.pt] (12.38,2.1339745962155607)-- (11.34877814529599,0.34784594999280527);
\draw [line width=2.pt] (15.,2.1339745962155607)-- (16.031221854704015,3.920103242438317);
\draw [line width=2.pt] (15.,2.1339745962155607)-- (16.031221854704015,0.347845949992805);
\draw [shift={(13.69,2.568398174392321)},line width=2.pt]  plot[domain=0.5235987755982988:2.617993877991494,variable=\t]({1.*2.7034101360919918*cos(\t r)+0.*2.7034101360919918*sin(\t r)},{0.*2.7034101360919918*cos(\t r)+1.*2.7034101360919918*sin(\t r)});
\draw [shift={(13.69,1.6995510180388012)},line width=2.pt]  plot[domain=0.5235987755982988:2.617993877991494,variable=\t]({1.*2.7034101360919918*cos(\t r)+0.*2.7034101360919918*sin(\t r)},{0.*2.7034101360919918*cos(\t r)+-1.*2.7034101360919918*sin(\t r)});
\draw (13.37720273937652,4.580902894449528) node[anchor=north west] {${{E_1}}$};
\draw (13.443500464048231,0.9212684925711269) node[anchor=north west] {${{E_2}}$};
\draw (10.393805129149559,2.698047513772959) node[anchor=north west] {${{E_3}}$};
\draw (16.546233978684274,2.737826148575985) node[anchor=north west] {${{E_4}}$};
\draw [line width=1.pt,dash pattern=on 2pt off 2pt] (15.,2.1339745962155607)-- (16.031221854704015,2.133974596215561);
\draw [line width=1.pt,dash pattern=on 2pt off 2pt] (16.031221854704015,3.920103242438317)-- (16.031221854704015,2.133974596215561);
\draw (15.29983675485612,2.18) node[anchor=north west] {${{y}}$};
\draw [line width=2.pt] (12.38,2.1339745962155607)-- (13.69,2.133974596215561);
\draw [line width=2.pt] (13.69,2.133974596215561)-- (15.,2.1339745962155607);
\draw (13.337424104573493,2.18) node[anchor=north west] {${{2x}}$};
\draw [fill=qqqqff] (12.38,2.1339745962155607) circle (2.5pt);
\draw [fill=qqqqff] (15.,2.1339745962155607) circle (2.5pt);
\draw [fill=qqqqff] (11.348778145295991,3.9201032424383175) circle (2.5pt);
\draw [fill=qqqqff] (11.34877814529599,0.34784594999280527) circle (2.5pt);
\draw [fill=qqqqff] (16.031221854704015,3.920103242438317) circle (2.5pt);
\draw [fill=qqqqff] (16.031221854704015,0.347845949992805) circle (2.5pt);
\end{tikzpicture}
  \caption{The competitor cluster defined in Proposition~\ref{prop:k_0}.}
  \label{fig:competitor}
\end{figure}

\begin{proof}
  Let
  $x \defeq 0.2707$,
  $y \defeq 0.394$
  and $R=2(x+y)/\sqrt{3}$. Consider the cluster represented in Figure~\ref{fig:competitor}.
  The area of the regions with four edges is given by:
  \[
  m(E_1) = m(E_2) = (2x+y) y \sqrt{3} + \frac{\pi}{3} R^2 - \frac{\sqrt 3}{4} R^2
  > 1
  \]
  while the area of the regions with three edges is:
  \[
  m(E_3) = m(E_4) = \sqrt{3} y^2 + \frac{\pi}{2} (y\sqrt{3})^2 > 1.
  \]
  So $\E\in\mathcal C^*(1,1,1,1)$. And we have
  \[
  P(\E) = 2 \frac{2\pi}{3} R + 2 \pi \sqrt{3} y + 2x + 8y
        \geq k_0.
  \]
\end{proof}

\begin{proposition}\label{prop:k_2}
  Let $\E \in \M^*(1,1,1,1)$ and suppose that $C$ is a
  component of some region.
  Then:
  \[
  m(C) \ge k_2 \defeq 0.0244.
  \]
  Moreover, if the number of regions which
  have an edge in common with $C$ is not larger than $3$ one has
  \[
  m(C) \ge k_6 \defeq 0.0425.
  \]
\end{proposition}

\begin{proof}
  We can apply Proposition~\ref{prop:varII} with
  $a_i=1$, $r\le 4$,
  $P(\E) \le  k_0$ so $r P(\E) \le \lambda \defeq 4 k_0$. We obtain:
  \[
  m(C)
  \ge \frac{\pi}{k_0^2}\left(1-\frac{\pi}{k_0^2}\right) \ge k_2.
  \]
  And with $r\le 3$ we would have
  \[
  m(C)
  \ge \frac{16 \pi}{9 k_0^2}\left(1-\frac{16 \pi}{9 k_0^2}\right) \ge k_6.
  \]
\end{proof}

\begin{proposition}\label{prop:k_1}
  Let $\E \in \M^*(1,1,1,1)$ be such that the region $E_1$ can be decomposed in two parts $E_1 = E_1' \cup C_1$ with
  \[
  m(E_1) = m(E_1') + m(C_1),\quad m(E_1') \ge m(C_1), \quad P(E_1) = P(E_1') + P(C_1)
  \]
  then
  \[
    m(C_1) \le k_1 \defeq  0.1605, \qquad
    P(C_1) \le k_7 \defeq  1.4199
  \]
\end{proposition}

\begin{proof}
  Let $m=m(C_1)$.
  By Lemma~\ref{lem:isop}, one has
  \[
  P(\E) \ge
  \sqrt{\pi}\left(\sqrt{4} + \sqrt{m} + \sqrt{m(E_1')} + 3 \sqrt{1}\right)
  = \sqrt{\pi}(\sqrt{m} + \sqrt{m(E_1')} + 5)
  \]
  whence
  \[
  \sqrt{m} + \sqrt{m(E_1')}
  \le \frac{P(\E)}{\sqrt{\pi}}-5
  \le \frac{k_0}{\sqrt{\pi}} - 5 \le c_1 \defeq 1.3168
  \]
  On one hand we have assumed that $m(E_1')\ge m(C_1)=m$, so
  $2\sqrt{m} \le c_1 < \sqrt 2$ which gives $m \le 1/2$.
  On the other hand we know that $m(E_1') = m(E_1)-m \ge 1-m$, whence
  \[
  \sqrt{m} + \sqrt{1-m} \le c_1.
  \]

  Now let $f(x) = \sqrt{x}+\sqrt{1-x}$.
  By computing the sign of $f'(x)$ we easily notice that $f(x)$
  is increasing for $x\in[0,1/2]$. By direct computation one checks
  that $f(k_1)>c_1$ (in fact $k_1$, which is defined in the statement of the
  theorem being proved, has been choosen to satisfy this relation).
  Since we know that $f(m)\leq c_1$ and $m\leq 1/2$ we
  conclude that $m = m(C_1)<k_1$.
  To get the estimate on the perimeter, we use again the isoperimetric
  inequality:
  \begin{align*}
  P(C_1) &= 2P(\E) - (P(E_1') + P(E_0) + \sum_{i=2}^4 P(E_i))\\
  &\le 2 k_0 - 2 \sqrt{\pi}(\sqrt{1-m(C_1)} + \sqrt{4} + 3 \sqrt{1})\\
  &\le 2 k_0 - 2 \sqrt{\pi}(\sqrt{1-k_1}+5) \le k_7.
  \end{align*}

\end{proof}

\begin{definition}[big/small components]\label{def:big_small}
  Let $\E$ be a regular $N$-cluster.
  We say that a component $C$ of a region $E_i$ is \emph{small} if
  $m(C)\le m(E_i)/2$. Otherwise we say that $C$ is \emph{big}. Notice
  that at most one connected component of each region can be big.
\end{definition}

\begin{corollary}\label{cor:one_big_component}
Let $\E\in \M^*(1,1,1,1)$. Then each region $E_i$ has exactly one big
connected component $E_i'$. Furthermore $m(E_i')\geq 1-k_1$,
where $k_1$ is the constant introduced in Proposition~\ref{prop:k_1}.
\end{corollary}
\begin{proof}
Let $E_i^1, \dots, E_i^M$ be the connected components of the region
$E_i$. Suppose by contradiction that all $E_i^j$ are small:
$m(E_i^j)\le m(E_i)/2$, $j=1,\dots,M$. Let $K$ be the smallest
index such that
\begin{equation}\label{eq:defK}
  \sum_{j=1}^K m(E_i^j) > k_1.
\end{equation}
We claim that
\begin{equation}\label{eq:mEik1}
\sum_{j=1}^K m(E_i^j) < m(E_i) - k_1.
\end{equation}
Otherwise we would have (notice that $k_1<1/4$)
\begin{align*}
\sum_{j=1}^{K-1} m(E_i^j) & = \sum_{j=1}^K m(E_i^j) - m(E_i^K)
\ge m(E_i) - k_1 - m(E_i^K) \\
& \ge m(E_i) - k_1 - \frac{m(E_i)}{2}
\ge \frac{m(E_i)}{2} - k_1 \\
& \ge \frac 1 2 - k_1 > k_1
\end{align*}
which is a contradiction since $K$ was the minimal index satisfying
the inequality~\eqref{eq:defK}.
So, if we define
\[
E_i' = \bigcup_{j=1}^K E_i^j,
\qquad
E_i'' = E_i \setminus E_i'
\]
we have (by \eqref{eq:defK} and \eqref{eq:mEik1})
\[
m(E_i') > k_1,
\qquad
m(E_i'') > k_1.
\]
This is now a contradiction with Proposition~\ref{prop:k_1}, since the smaller
of the two components $E_i'$, $E_i''$ should have a measure smaller
than $k_1$.

Finally if $E_i'$ is the big connected component of the region $E_i$,
applying Proposition~\ref{prop:k_1} with $C_i = E_i\setminus E_i'$, we find
$m(E_i)\geq 1-k_1$.
\end{proof}

\begin{corollary}\label{cor:one_big_internal}
  Let $\E \in \M^*(1,1,1,1)$. Then at most one of the big
  components is internal.
\end{corollary}
\begin{proof}
  Suppose by contradictions that two big components $E_i^1$ and
  $E_j^1$ are internal. Then by the isoperimetric inequality:
  \begin{align*}
    P(\E)
    &\ge P(E_i^1\cup E_i^2) + P(E_0) \\
    & \ge 2 \sqrt{\pi}\left(\sqrt{m(E_i^1) + m(E_i^2)} +
      \sqrt{m(E_1)+m(E_2)+m(E_3)+m(E_4)}\right)\\
      &  \ge 2 \sqrt{\pi}\left(\sqrt{2(1-k_1)}+\sqrt{4}\right)
      \ge 11.6831 > k_0 \geq p^*(1,1,1,1).
  \end{align*}
  Which is a contradiction.
\end{proof}

\begin{proposition}\label{prop:k_3}
  Let $\E \in \M^*(1,1,1,1)$ be such that both regions $E_i$ and $E_j$
  are disconnected ($i\neq j$). Then every small component $C$ of
  either $E_i$ or $E_j$ satisfies:
  \[
    m(C) \le k_3 \defeq 0.0408, \qquad
    P(C) \le k_9 \defeq 0.7154.
  \]
\end{proposition}

\begin{proof}
Without loss of generality we might suppose that $i=1$, $j=2$. Let
$E_1'$ be the larger small component of $E_1$ and let $E_2'$ be the
larger small component of $E_2$. Suppose moreover that $m\defeq m(E_1')\ge
m(E_2')$. Then we have
\begin{align*}
m(E_1\setminus E_1') &\ge 1 - m, &
m(E_1') &= m,\\
m(E_2\setminus E_2') &\ge 1 - m, &
m(E_2') &\ge k_2.
\end{align*}
So, from the isoperimetric inequality:
\begin{align*}
  \frac{P(\E)}{\sqrt{\pi}}
  &\ge \sqrt{m(\RR^2\setminus E_0)} + \sum_{i=1}^2
  \sqrt{m(E_i\setminus E_i')} + \sum_{i=1}^2 \sqrt{m(E_i')} +
  \sum_{j=3}^4 \sqrt{m(E_j)}
\end{align*}
we obtain:
\begin{equation}\label{eq:fmk2}
  \begin{aligned}
  \frac{P(\E)}{\sqrt{\pi}}
  & \ge \sqrt{4} + \sqrt{1-m} + \sqrt{m} + \sqrt{1-m} + \sqrt{k_2} +
  2\sqrt{1}\nonumber\\
  & = 4 + 2\sqrt{1-m} + \sqrt{m} + \sqrt{k_2}.
  \end{aligned}
\end{equation}
If we set $f(x) = 2\sqrt{1-x} + \sqrt{x}$ and remember that $P(\E)\le k_0$
(Proposition~\ref{prop:k_0}) we obtain
\[
 f(m) \le \frac{k_0}{\sqrt{\pi}} - 4 - \sqrt{k_2} \le c_2 \defeq
 2.1606
\]
We have:
\[
f'(x) = -(1-x)^{-\frac 1 2} + \frac 1 2 x^{-\frac 1 2},\quad
f''(x) = -\frac 1 2 (1-x)^{-\frac 3 2} -\frac 1 4 x^{-\frac 3 2}.
\]
By direct computation one checks that $f'(k_1) > 0.1565 > 0$ and since $f''<0$ we
know that $f$ is strictly increasing on $[0,k_1]$. By direct
computation one checks $k_3$ was choosen so that $f(k_3) > c_2$.
If, by contradiction, $m> k_3$ since $m\in[k_2, k_1]$ (by Proposition~\ref{prop:k_2} and Proposition~\ref{prop:k_1}) we would have
$f(m) > f(k_3) > c_2$ against~\eqref{eq:fmk2}. So $m<k_3$.

Since $m$ was the measure of the largest small component we obtain the
first estimate: $m(C)\le m \leq k_3$.

To prove the estimate on the perimeter $P(C)$ suppose now that
$C=E_1'$ (not it will not matter if $E_1'$ is larger or smaller than
$E_2'$).
Recall that (Proposition~\ref{prop:k_2})
\[
m(E_1') \ge k_2, \qquad
m(E_2') \ge k_2
\]
and the previous estimate gives:
\[
m(E_1\setminus E_1') \ge 1 - k_3, \qquad
m(E_2\setminus E_2') \ge 1 - k_3.
\]
Hence, using the isoperimetric inequality we have
\begin{align*}
2 P(\E) & = P(E_1') + P(E_2') + \sum_{i=1}^2 P(E_i\setminus E_i') +
\sum_{i=3}^4 P(E_i) + P(\RR^2 \setminus E_0)\\
 & \ge P(E_1') + 2 \sqrt{\pi}\left(\sqrt{k_2} + 2 \sqrt{1-k_3} + 2
\sqrt{1} + \sqrt{4}\right)
\end{align*}
whence, recalling also that $P(\E)\le k_0$:
\[
P(E_1') \le 2 k_0 - 2 \sqrt{\pi}(\sqrt{k_2} + 2\sqrt{1-k_3} + 4) \le
k_9.
\]
\end{proof}

\begin{proposition}\label{prop:k_4}
  Let $\E \in \M^*(1,1,1,1)$ be such that the region $E_i$ has at
  least three components.
  Then every small component $C$ of
  $E_i$ satisfies:
  \[
  m(C) \le k_4 \defeq 0.0411.
  \]
\end{proposition}

\begin{proof}
Without loss of generality we might suppose that $i=1$.
Notice that, by Corollary~\ref{cor:one_big_component}, there are
 at least two small components of $E_1$. Let
$E_1'$ be the larger small component of $E_1$ and $E_1''$ be another
small component of $E_1$. Let $m \defeq m(E_1') \ge m(E_1'')$.
Then we have
\begin{align*}
m(E_1\setminus (E_1'\cup E_1'')) &\ge 1 - m - m,\qquad
m(E_1') = m, \qquad
m(E_1'') \ge k_2.
\end{align*}
So, from the isoperimetric inequality:
\begin{align*}
  \frac{P(\E)}{\sqrt \pi}
  &\ge \sqrt{m(\RR^2\setminus E_0)} +
  \sqrt{m(E_i\setminus (E_i'\cup E_i''))} \\
  &\quad + \sqrt{m(E_i')} + \sqrt{m(E_i'')}
  + \sum_{j=2}^4 \sqrt{m(E_j)}
\end{align*}
we obtain:
\begin{equation}\label{eq:fmk2bis}
  \begin{aligned}
  \frac{P(\E)}{\sqrt{\pi}}
  & \ge \sqrt{4} + \sqrt{1-2m} + \sqrt{m} + \sqrt{k_2} +
  3\sqrt{1}\\
  & = 5 + \sqrt{1-2m} + \sqrt{m} + \sqrt{k_2}.
  \end{aligned}
\end{equation}
If we set $f(x) = \sqrt{1-2x} + \sqrt{x}$ and remember that $P(\E)\le k_0$
(Proposition~\ref{prop:k_0}) we obtain
\[
 f(m) \le \frac{k_0}{\sqrt{\pi}} - 5 - \sqrt{k_2} \le c_3 \defeq 1.1606
 \]
We have:
\[
f'(x) = -(1-2x)^{-\frac 1 2} + \frac 1 2 x^{-\frac 1 2},\quad
f''(x) = -(1-2x)^{-\frac 3 2} -\frac 1 4 x^{-\frac 3 2}.
\]
By direct computation one checks that $f'(k_1) > 0.0344 >0$ and since $f''<0$ we
know that $f$ is strictly increasing on $[0,k_1]$. By direct
computation one checks that $k_4$ has been choosen so that $f(k_4) > c_3$.
If, by contradiction, $m> k_4$ since $m\in[k_2, k_1]$
(by Proposition~\ref{prop:k_2} and Proposition~\ref{prop:k_1}) we would have
$f(m) > f(k_4) > c_3$ against~\eqref{eq:fmk2bis}. So $m\leq k_4$.

\end{proof}

\begin{proposition}\label{prop:two_small}
  Let $\E \in \M^*(1,1,1,1)$. Then the total number of small
  components is not larger than two.
\end{proposition}

\begin{proof}
Suppose by contradiction that the cluster $\E\in \M^*(1,1,1,1)$ has at
least three small components $C_1, C_2, C_3$.
Suppose
$m\defeq m(C_1) \ge m(C_2) \ge m(C_3)$.
Let $C=C_1\cup C_2 \cup C_3$ and let
$E_i' = E_i\setminus C$ for $i=1,\dots,4$.

From the isoperimetric inequality:
\begin{align*}
  \frac{P(\E)}{\sqrt \pi}
  &\ge \sqrt{m(\RR^2\setminus E_0)} + \sum_{i=1}^4 \sqrt{m(E_i')} + \sum_{i=1}^3 \sqrt{m(C_i)}.
\end{align*}
Now consider the quantity
\[
A = \sum_{i=1}^4 \sqrt{m(E_i')}
\]
to get an estimate of $A$ from below we use the estimates
$k_2 \le m(C_i) \le m$
but we have to distinguish three different cases:
\begin{enumerate}
\item if the small components
  all belong to the same region we have $A \ge
  \sqrt{1-3m} + 3 \sqrt 1$;
\item if only two of the small components belong to the same region:
  $A \ge \sqrt{1-2m} + \sqrt{1-m} + 2\sqrt 1$;
\item if the three small components belong to three
     different regions: $A \ge 3\sqrt{1-m} + \sqrt 1$.
\end{enumerate}

With a straightforward algebraic manipulation one can check that for
all $x\in [0, 1/3]$ one has
\[
3 \sqrt{1-x} + 1 \ge \sqrt{1-2x}+\sqrt{1-x} + 2 \ge \sqrt{1-3x} + 3
\]
so that in every case it holds $A \ge \sqrt{1-3m} + 3$.




Hence
\begin{equation}\label{eq:fmk2tris}
\begin{aligned}
    \frac{P(\E')}{\sqrt{\pi}}
  & \ge \sqrt{4} + \sqrt{1-3m} + 3 + \sqrt{m} + 2\sqrt{k_2} \\
  & = \sqrt{1-3m} + \sqrt{m} + 5 + 2\sqrt{k_2}
\end{aligned}
\end{equation}
If we set $f(x) = \sqrt{1-3x} + \sqrt{x}$ and remember that $P(\E') =
P(\E) \le k_0$
(Proposition~\ref{prop:k_0}) we obtain
\[
 f(m) \le \frac{k_0}{\sqrt{\pi}} - 5 - 2\sqrt{k_2} \le k_5 \defeq 1.0044
 \]
We have:
\[
f'(x) = -\frac 3 2 (1-3x)^{-\frac 1 2} + \frac 1 2 x^{-\frac 1 2},\quad
f''(x) = -\frac 9 4 (1-x)^{-\frac 3 2} -\frac 1 4 x^{-\frac 3 2}.
\]
By direct computation one checks that $f(k_1) > 1.1206 > k_5$ and
$f(k_2)> 1.1189 >k_5$. And since $f''<0$ we
know that $f$ is concave and hence $f(x)>k_5$ if $x\in[k_2,k_1]$. Since
$f(m)\le k_5$ and we already know that $m\ge k_2$ (Proposition~\ref{prop:k_2})
we conclude that $m>k_1$, which is a contradiction.
\end{proof}

\begin{corollary}\label{cor:two_small}
  Let $\E\in \M^*(1,1,1,1)$.
  Then there are at most six bounded connected components.
  Four connected components are big
  and at most two are small (see Definition~\ref{def:big_small}).

  If the small components are exactly two,
  they
  have measure between $k_2$ and $k_4$, they are external, and
  they have edges in common with all the other regions. If the two
  small components belong to the same region they both have four edges,
  while if they belong to different regions they might have four or five edges.

  If there is only one small component
  it has measure not larger than $k_1$.
\end{corollary}
\begin{proof}
  By Proposition~\ref{prop:two_small} there are at most two small components, so the total
  number of bounded connected components is at most six.

  If we have two small
  components they
  can either belong to the same region, and then by
  Proposition~\ref{prop:k_4} each small component has measure not larger than $k_4$. Or, the
  two components belong to different regions and then by
  Proposition~\ref{prop:k_3} each small component has measure not larger than
  $k_3<k_4$. Every small component which is adjacent only to three
  other regions would have measure
  larger than $k_6$ by Proposition~\ref{prop:k_2} and since $k_6 > k_4$ this is
  impossible. So every small component must have edges in common with
  all the other four regions, included the external one: so they have
  at least four edges and are external.
  If the two components belong
  to two different regions they can have four or five edges (the two
  small component might have an edge in common).
  If the two components belong to the same region, each other region is
  connected and hence they cannot have more than four edges (each edge
  is adjacent to a different component).

  If there is only one small component we can only apply
  Proposition~\ref{prop:k_1} to get the estimate with the constant $k_1$.
\end{proof}

\section{Clusters with six components}\label{sec:six}

In this section we will consider possible
minimizers $\E\in\M^*(1,1,1,1)$ with exactly six bounded
components and we will exclude that they exist.

The following Corollary assures that we have $m(E_i)=1$ for
$i=1,\dots,4$. This will be used in the following without further notice.

\begin{corollary}\label{corollary:strong_weak}
$\M^*(1,1,1,1) = \M(1,1,1,1)$.
\end{corollary}

\begin{proof}
Given any $\E\in \M^*(1,1,1,1)$
by Corollary~\ref{cor:two_small} we know that $\E$
has no more than six bounded components.
By Proposition~\ref{prop:edges} we conclude that
$\E\in \M(1,1,1,1)$, hence $\M^*(1,1,1,1)\subset \M(1,1,1,1)$.
Since $\M^*(1,1,1,1)$ is not empty (Theorem~\ref{thm:weak_existence_regularity})
we obtain $p^*(1,1,1,1) = P(\E) = p(1,1,1,1)$.

On the other hand, given $\E'\in\M(1,1,1,1)$ we have
$\E'\in\mathcal C^*(1,1,1,1)$
and since $P(\E') = p(1,1,1,1) = p^*(1,1,1,1)$ we
conclude that $\E'\in \M^*(1,1,1,1)$.
\end{proof}

\begin{corollary}\label{cor:not_three}
  Let $\E\in \M^*(1,1,1,1)$. Then we exclude that one region $E_i$ can
  have three components.
\end{corollary}
\begin{proof}
  Suppose by contradiction that
  the region $E_1$ is composed by three components: one big and two
  small
  (recall that, by Corollary~\ref{cor:one_big_component},
  each region has one big component).
   By Proposition~\ref{prop:edges} we know that every component has at least three
  edges.
  By Corollary~\ref{cor:two_small}, a small component has four edges,
   so, the two small components have exactly
  four vertices and the region $E_1$ has at least $3+4+4 = 11$ vertices.
   But the total number of bounded connected components is
  $M=6$ and by Proposition~\ref{prop:edges} the number of vertices should be
  $v = 2(M-1) = 10$. This is a contradiction.
\end{proof}

\begin{proposition}\label{prop:not_two_two}
  Let $\E \in \M^*(1,1,1,1)$. Then we exclude that two different
  regions are disconnected.
\end{proposition}
\begin{proof}
  By contradiction suppose that $C_1$ and $C_2$ are small components of $E_1$ and $E_2$
  respectively and let $E_1' = E_1\setminus C_1$ and
  $E_2' = E_2\setminus C_2$ be the two big components.

  Recall that, by Corollary~\ref{cor:two_small}, the small components $C_1$ and $C_2$
  have four or five edges.

  If the component $C_i$ ($i=1,2$) has five edges,
  by Proposition~\ref{prop:pressure_perimeter} and Proposition~\ref{prop:k_3},
  one finds that
  \begin{equation}\label{pi_four}
  p_i \ge \frac{\pi}{3 P(C)} \ge \frac{\pi}{3 k_9} > 1.4637 > \frac{k_0}{8}
  \end{equation}
  On the other hand if $C_i$ has only four edges, one finds:
  \[
  p_i \ge \frac{2\pi}{3 P(C)} \ge \frac{2\pi}{3 k_9} > \frac{k_0}{4}.
  \]



  Remember that, by Theorem~\ref{thm:pressure_formula} and Proposition~\ref{prop:k_0}, we have
  \[
  p_1 + p_2 + p_3 + p_4 = \frac{P(\E)}{2} \le \frac{k_0}{2}.
  \]
  Without loss of generality we might and shall suppose that
  $p_1\ge p_2$.

  Notice that $p_1$ and $p_2$ are both larger than the average
  and, in
  particular, $p_2$ is not the lowest pressure: $p_2 > \min\{p_3,p_4\}$.
  If both regions $C_1$ and $C_2$ had four edges, we
  would find $p_1+p_2 > k_0 / 2$ which is a contradiction. Hence we
  know that $C_1$ has four or five edges and $C_2$ has five edges (if
  $C_i$ has four edges $p_i$ is the higher pressure).

  \emph{Step 1:} we claim that at most one component is internal.
  By Corollary~\ref{cor:two_small} we know that the small components are external
  and by Corollary~\ref{cor:one_big_internal} we know that at most one big component is
  internal. The claim follows.

  \emph{Step 2:} we claim that $E_2'$ is external and has three or
  four edges.

  Notice that since at most one component is internal, and we have a
  total of 6 bounded components, the external region $E_0$ has
  either 5 or 6 vertices. On the other hand the big component $E_2'$ has
  at least 3 vertices and the small component $C_2$ has 5
  vertices. Two of the vertices of $C_2$ are in common with the
  vertices of $E_0$ and, if $E_2'$ were internal, all its vertices
  would be distinct from the vertices of $E_0$ and, of course, from
  the vertices of $C_2$. So we find at least $3+3+5 = 11$ distinct
  vertices of
  the cluster $\E$ while we know (Proposition~\ref{prop:edges}) that $\E$ has exactly
  $10$ vertices.

  The same contradiction holds in the case that $E_2'$ has more than four
  vertices since also in this case at least three of them would be
  internal.

  \emph{Step 3:} we claim that $E_1'$ and $E_2'$ are
  adjacent. Let $\ell_1$ and $\ell_2$ be the lengths of the
  external edges of $E_1'$ and $E_2'$ respectively
  ($\ell_i=0$ if $E_i'$ is internal). Suppose by contradiction that
  $E_1'$ and $E_2'$ have no common edge. Then
  \[
   k_0 \ge P(\E) \ge P(E_1') + P(E_2') + P(E_0) - (\ell_1+\ell_2)
   \]
   and by applying the isoperimetric inequality and the estimates
   $m(E_i')\ge 1- k_3$ we obtain:
   \[
   k_0 \ge 2\sqrt{\pi}(2\sqrt{1-k_3}+\sqrt{4}) - (\ell_1+\ell_2)
   \]
   whence
   \[
   \frac{\ell_1 + \ell_2}{2} \ge 2 \sqrt{\pi}(\sqrt{1-k_3}+1) - \frac{k_0}{2} > c_4 \defeq 1.4186.
   \]
   If we let $\ell_i$ be the largest between $\ell_1$ and $\ell_2$ we have
   $\ell_i > c_4$
   and from Proposition~\ref{prop:pressure_length} we obtain the following estimate on the
   pressure of the corresponding region $E_i$ (remember that every
   component of $\E$ is adjacent to at most four different regions):
   \begin{equation}\label{pi_more}
   p_i
   \ge \frac{\sqrt{\pi}}{2\sqrt{m(C_i)}} - \frac{2}{\ell_i}
   \ge \frac{\sqrt{\pi}}{2\sqrt{k_3}} - \frac{2}{c_4}
   > 2.9776 > \frac{k_0}{4}.
   \end{equation}
   Remember that $p_1+p_2+p_3+p_4 \le k_0/2$ so $p_i$ is the highest
   pressure (actually $i=1$ since we decided that $p_1\ge p_2$).
   Then let $n\ge 3$ be the number of edges of
   $E_i'$ and let $L_{i,j}$ be the total length of the edges in common
   between $E_i'$ and $E_j$ (so that $L_{i,0}=\ell_i$):
   \[
   \pi \ge \frac{(6-n)\pi}{3} = \sum_j (p_i - p_j) L_{i,j}
   \ge p_i \ell_i
   \]
   whence:
   \[
    p_i \le \frac{\pi}{\ell_i} \le \frac{\pi}{c_4} < 2.2146
   \]
   which is in contradiction with with \eqref{pi_more}.

\emph{Step 4}: if a connected region $E_i$ ($i=3,4$) is internal, it is
adjacent to both $E_1'$ and $E_2'$.

The proof is the same as in the previous Step. Just take $E_i$ in place of $E_2'$
and $E_1'$ or $E_2'$ in place of $E_1'$. Notice that $\ell_2=0$ so that
$\ell_i = \ell_1$ and the proof completes in exactly the same way (the
estimates are actually stronger).


\emph{Step 5}: we claim that if one of $E_3$ or $E_4$ is internal and the other
one is external with only three edges, then $E_3$ and $E_4$ must be
adjacent. We proceed in a similar way as the step before. Suppose by
contradiction that $E_3$ is internal and not adjacent to $E_4$.

So
$E_3$ is only adjacent to the components of $E_1$ and $E_2$ and it has at most four edges, so,
by Lemma~\ref{lem:turning_angle}, we have
\[
0 < \frac{(6-4)\pi}{3} \le \sum_{i=1}^2 (p_3-p_i)L_{3,i}.
\]
We deduce that $p_3\ge p_2$ since otherwise (being $p_1\ge p_2$) the
right hand side of the previous  equation would be negative. So
$p_3\ge p_2\ge k_0/8$.

Now, let
$\ell_i$ be the length of the external edges of $E_i'$ (recall that
only one component can be internal hence $E_i'$ is external and
$\ell_i>0$).
We have
\[
k_0 \ge P(\E) \ge P(E_1' \cup E_2' \cup E_3) + P(E_0) - (\ell_1+\ell_2)
\]
whence, by applying the isoperimetric inequality,
\[
\frac{\ell_1+\ell_2}{2} \ge \sqrt{\pi}(\sqrt{2(1-k_3)+1} + \sqrt{4}) -
\frac{k_0}{2}
> c_5 \defeq 0.9747.
\]
Now if $\ell_i$ is the maximum between $\ell_1$ and $\ell_2$ we know
that $\ell_i > c_5$.
By Proposition~\ref{prop:pressure_length} (since any component can be
adjacent to at most 4 different regions), we have
\[
p_i \ge \frac{2\sqrt{\pi}}{4\sqrt{m(C_i)}}- \frac{2}{\ell_i}
\ge \frac{\sqrt\pi}{2\sqrt{k_3}} - \frac{2}{c_5} > 2.3355 >
\frac{3}{16} k_0.
\]
So $p_1 > 3 k_0 / 16$ (since $p_1$ has been choosen to be the maximum between
$p_1$ and $p_2$).

Now we work on $E_4$ which is external with $m=3$ edges. Remember that
$p_2$ cannot be the lowest pressure and since $p_1\ge p_2$ and
$p_3\ge p_2$ we deduce that $p_4$ is the lowest pressure.
Hence, by
Lemma~\ref{lem:turning_angle}
\[
\pi = \frac{(6-m)\pi}{3} = \sum_j (p_4-p_j) L_{4,j} \le p_4 L_{4,0}
\]
and by Proposition~\ref{prop:varI}
\[
p_4 \ge \frac{\pi}{L_{4,0}} \ge \frac{\pi}{2\sqrt{\pi}\sqrt{m(E_4)}} =
\frac{\sqrt{\pi}}{2} > 0.8862 > \frac{k_0}{16}
\]
So, we have found that
\[
 P(\E) = 2(p_1 + p_2 + p_3 + p_4) > 2\left(\frac{3k_0}{16} + \frac{k_0}{8} +
 \frac{k_0}{8} + \frac{k_0}{16}\right) = k_0
\]
which contradicts the minimality of $\E$. The claim is proved.

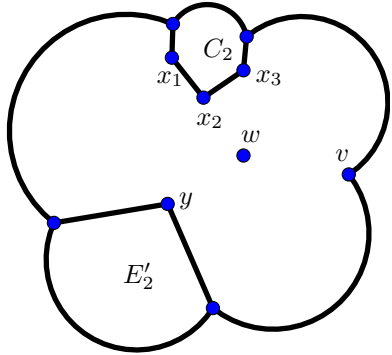
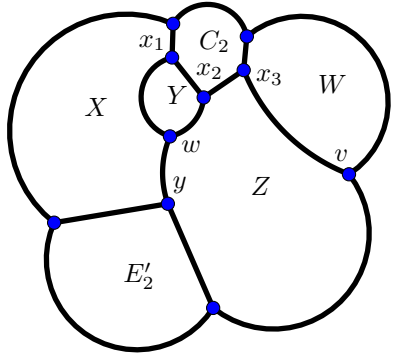
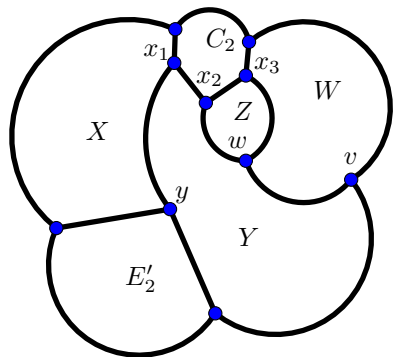
\begin{figure}
  \centering \subfigure[Case (i), incomplete diagram]{
\begin{tikzpicture}[line cap=round,line join=round,>=triangle 45,x=0.7cm,y=0.7cm]
\clip(-3.6,-1.1) rectangle (3.7,5.7);
\draw[line width=2pt] (-0.42,5.22)-- (-0.44,4.58);
\draw[line width=2pt] (-0.44,4.58)-- (0.16,3.82);
\draw[line width=2pt] (0.16,3.82)-- (0.92,4.34);
\draw[line width=2pt] (0.92,4.34)-- (0.98,4.98);
\draw[line width=2pt] [shift={(0.23317073170731736,4.826829268292684)}] plot[domain=0.20228948488978113:2.599746620763336,variable=\t]({1.*0.7623747300575605*cos(\t r)+0.*0.7623747300575605*sin(\t r)},{0.*0.7623747300575605*cos(\t r)+1.*0.7623747300575605*sin(\t r)});
\draw[line width=2pt] [shift={(-1.3065030674846625,3.184417177914111)}] plot[domain=1.160063019204046:4.045398549712445,variable=\t]({1.*2.220244381646054*cos(\t r)+0.*2.220244381646054*sin(\t r)},{0.*2.220244381646054*cos(\t r)+1.*2.220244381646054*sin(\t r)});
\draw[line width=2pt] [shift={(-1.1099782055938978,0.7418924809298947)}] plot[domain=2.723198108469533:5.7168591671619815,variable=\t]({1.*1.718232389140763*cos(\t r)+0.*1.718232389140763*sin(\t r)},{0.*1.718232389140763*cos(\t r)+1.*1.718232389140763*sin(\t r)});
\draw[line width=2pt] [shift={(1.4825745220311104,1.2397471390392654)}] plot[domain=-2.2484410191826436:0.662020010267545,variable=\t]({1.*1.822404586584658*cos(\t r)+0.*1.822404586584658*sin(\t r)},{0.*1.822404586584658*cos(\t r)+1.*1.822404586584658*sin(\t r)});
\draw[line width=2pt] [shift={(2.030801134349521,3.7298298475717826)}] plot[domain=-0.9950323237592027:2.269764756630019,variable=\t]({1.*1.6331284193144862*cos(\t r)+0.*1.6331284193144862*sin(\t r)},{0.*1.6331284193144862*cos(\t r)+1.*1.6331284193144862*sin(\t r)});
\draw[line width=2pt] (-2.68,1.44)-- (-0.52,1.8);
\draw[line width=2pt] (-0.52,1.8)-- (0.34,-0.18);
\draw (-0.02,5.12) node[anchor=north west] {$C_2$};
\draw (-1.54,0.86) node[anchor=north west] {$E_2'$};
\draw (-0.9,4.5) node[anchor=north west] {$x_1$};
\draw (-0.15,3.7) node[anchor=north west] {$x_2$};
\draw (0.98,4.52) node[anchor=north west] {$x_3$};
\draw (-0.48,2.2) node[anchor=north west] {$y$};
\draw [fill=qqqqff] (-0.42,5.22) circle (2.5pt);
\draw [fill=qqqqff] (-0.44,4.58) circle (2.5pt);
\draw [fill=qqqqff] (0.16,3.82) circle (2.5pt);
\draw [fill=qqqqff] (0.92,4.34) circle (2.5pt);
\draw [fill=qqqqff] (0.98,4.98) circle (2.5pt);
\draw [fill=qqqqff] (-2.68,1.44) circle (2.5pt);
\draw [fill=qqqqff] (0.34,-0.18) circle (2.5pt);
\draw [fill=qqqqff] (2.92,2.36) circle (2.5pt);
\draw (2.8,2.72) node {$v$};
\draw [fill=qqqqff] (-0.52,1.8) circle (2.5pt);
\draw [fill=qqqqff] (0.92,2.72) circle (2.5pt);
\draw (1.06,3.08) node {$w$};
\end{tikzpicture}    \label{Fig_02_a.tikz}}\\
  \subfigure[]{
    \begin{tikzpicture}[line cap=round,line join=round,>=triangle 45,x=0.7cm,y=0.7cm]
\clip(-3.6,-1.1) rectangle (3.7,5.7);
\draw[line width=2pt] (-0.42,5.22)-- (-0.44,4.58);
\draw[line width=2pt] (-0.44,4.58)-- (0.16,3.82);
\draw[line width=2pt] (0.16,3.82)-- (0.92,4.34);
\draw[line width=2pt] (0.92,4.34)-- (0.98,4.98);
\draw[line width=2pt] [shift={(0.23317073170731736,4.826829268292684)}] plot[domain=0.20228948488978113:2.599746620763336,variable=\t]({1.*0.7623747300575605*cos(\t r)+0.*0.7623747300575605*sin(\t r)},{0.*0.7623747300575605*cos(\t r)+1.*0.7623747300575605*sin(\t r)});
\draw[line width=2pt] [shift={(-1.3065030674846625,3.184417177914111)}] plot[domain=1.160063019204046:4.045398549712445,variable=\t]({1.*2.220244381646054*cos(\t r)+0.*2.220244381646054*sin(\t r)},{0.*2.220244381646054*cos(\t r)+1.*2.220244381646054*sin(\t r)});
\draw[line width=2pt] [shift={(-1.1099782055938978,0.7418924809298947)}] plot[domain=2.723198108469533:5.7168591671619815,variable=\t]({1.*1.718232389140763*cos(\t r)+0.*1.718232389140763*sin(\t r)},{0.*1.718232389140763*cos(\t r)+1.*1.718232389140763*sin(\t r)});
\draw[line width=2pt] [shift={(1.4825745220311104,1.2397471390392654)}] plot[domain=-2.2484410191826436:0.662020010267545,variable=\t]({1.*1.822404586584658*cos(\t r)+0.*1.822404586584658*sin(\t r)},{0.*1.822404586584658*cos(\t r)+1.*1.822404586584658*sin(\t r)});
\draw[line width=2pt] [shift={(2.030801134349521,3.7298298475717826)}] plot[domain=-0.9950323237592027:2.269764756630019,variable=\t]({1.*1.6331284193144862*cos(\t r)+0.*1.6331284193144862*sin(\t r)},{0.*1.6331284193144862*cos(\t r)+1.*1.6331284193144862*sin(\t r)});
\draw[line width=2pt] (-2.68,1.44)-- (-0.52,1.8);
\draw[line width=2pt] (-0.52,1.8)-- (0.34,-0.18);
\draw (-0.1,5.27) node[anchor=north west] {$C_2$};
\draw (-1.54,0.86) node[anchor=north west] {$E_2'$};
\draw (-1.25,5.18) node[anchor=north west] {$x_1$};
\draw (-0.16,4.64) node[anchor=north west] {$x_2$};
\draw (0.98,4.52) node[anchor=north west] {$x_3$};
\draw (-0.6,2.5) node[anchor=north west] {$y$};
\draw (2.47,3.0) node[anchor=north west] {$v$};
\draw (-0.44,3.18) node[anchor=north west] {$w$};
\draw[line width=2pt] [shift={(1.1533333333333344,2.3883333333333328)}] plot[domain=2.741019162485899:3.47968647783315,variable=\t]({1.*1.7737476019872611*cos(\t r)+0.*1.7737476019872611*sin(\t r)},{0.*1.7737476019872611*cos(\t r)+1.*1.7737476019872611*sin(\t r)});
\draw[line width=2pt] [shift={(-0.2493178588346757,3.824381809568925)}] plot[domain=1.8179878812509616:4.411876729179429,variable=\t]({1.*0.7793064407983076*cos(\t r)+0.*0.7793064407983076*sin(\t r)},{0.*0.7793064407983076*cos(\t r)+1.*0.7793064407983076*sin(\t r)});
\draw[line width=2pt] [shift={(-0.6599236641221362,3.8823664122137402)}] plot[domain=4.932981136357144:6.207267812284321,variable=\t]({1.*0.8222921526804721*cos(\t r)+0.*0.8222921526804721*sin(\t r)},{0.*0.8222921526804721*cos(\t r)+1.*0.8222921526804721*sin(\t r)});
\draw[line width=2pt] [shift={(4.19509615384616,5.648076923076929)}] plot[domain=3.5215826090441795:4.342449191591928,variable=\t]({1.*3.526658482703467*cos(\t r)+0.*3.526658482703467*sin(\t r)},{0.*3.526658482703467*cos(\t r)+1.*3.526658482703467*sin(\t r)});
\draw (-0.72,4.22) node[anchor=north west] {$Y$};
\draw (2.16,4.42) node[anchor=north west] {$W$};
\draw (-2.28,3.98) node[anchor=north west] {$X$};
\draw (0.88,2.48) node[anchor=north west] {$Z$};
\begin{scriptsize}
\draw [fill=qqqqff] (-0.42,5.22) circle (2.5pt);
\draw [fill=qqqqff] (-0.44,4.58) circle (2.5pt);
\draw [fill=qqqqff] (0.16,3.82) circle (2.5pt);
\draw [fill=qqqqff] (0.92,4.34) circle (2.5pt);
\draw [fill=qqqqff] (0.98,4.98) circle (2.5pt);
\draw [fill=qqqqff] (-2.68,1.44) circle (2.5pt);
\draw [fill=qqqqff] (0.34,-0.18) circle (2.5pt);
\draw [fill=qqqqff] (2.92,2.36) circle (2.5pt);
\draw [fill=qqqqff] (-0.52,1.8) circle (2.5pt);
\draw [fill=qqqqff] (-0.48,3.08) circle (2.5pt);
\end{scriptsize}
\end{tikzpicture}

    \label{Fig_02_a_1.tikz}}
  \subfigure[]{
    \begin{tikzpicture}[line cap=round,line join=round,>=triangle 45,x=0.7cm,y=0.7cm]
\clip(-3.6,-1.00) rectangle (3.70,5.70);
\draw[line width=2pt] (-0.42,5.22)-- (-0.44,4.58);
\draw[line width=2pt] (-0.44,4.58)-- (0.16,3.82);
\draw[line width=2pt] (0.16,3.82)-- (0.92,4.34);
\draw[line width=2pt] (0.92,4.34)-- (0.98,4.98);
\draw[line width=2pt] [shift={(0.23317073170731736,4.826829268292684)}] plot[domain=0.20228948488978113:2.599746620763336,variable=\t]({1.*0.7623747300575605*cos(\t r)+0.*0.7623747300575605*sin(\t r)},{0.*0.7623747300575605*cos(\t r)+1.*0.7623747300575605*sin(\t r)});
\draw[line width=2pt] [shift={(-1.3065030674846625,3.184417177914111)}] plot[domain=1.160063019204046:4.045398549712445,variable=\t]({1.*2.220244381646054*cos(\t r)+0.*2.220244381646054*sin(\t r)},{0.*2.220244381646054*cos(\t r)+1.*2.220244381646054*sin(\t r)});
\draw[line width=2pt] [shift={(-1.1099782055938978,0.7418924809298947)}] plot[domain=2.723198108469533:5.7168591671619815,variable=\t]({1.*1.718232389140763*cos(\t r)+0.*1.718232389140763*sin(\t r)},{0.*1.718232389140763*cos(\t r)+1.*1.718232389140763*sin(\t r)});
\draw[line width=2pt] [shift={(1.4825745220311104,1.2397471390392654)}] plot[domain=-2.2484410191826436:0.662020010267545,variable=\t]({1.*1.822404586584658*cos(\t r)+0.*1.822404586584658*sin(\t r)},{0.*1.822404586584658*cos(\t r)+1.*1.822404586584658*sin(\t r)});
\draw[line width=2pt] [shift={(2.030801134349521,3.7298298475717826)}] plot[domain=-0.9950323237592027:2.269764756630019,variable=\t]({1.*1.6331284193144862*cos(\t r)+0.*1.6331284193144862*sin(\t r)},{0.*1.6331284193144862*cos(\t r)+1.*1.6331284193144862*sin(\t r)});
\draw[line width=2pt] (-2.68,1.44)-- (-0.52,1.8);
\draw[line width=2pt] (-0.52,1.8)-- (0.34,-0.18);
\draw (-0.017504737524938748,5.40) node[anchor=north west] {$C_2$};
\draw (-1.5467112851369347,0.9071848447599059) node[anchor=north west] {$E_2'$};
\draw (-1.2,5.1) node[anchor=north west] {$x_1$};
\draw (-0.2,4.60) node[anchor=north west] {$x_2$};
\draw (0.9,4.85) node[anchor=north west] {$x_3$};
\draw (-0.60,2.4) node[anchor=north west] {$y$};
\draw (0.40,3.36) node[anchor=north west] {$w$};
\draw (2.6,3.05) node[anchor=north west] {$v$};
\draw[line width=2pt] [shift={(1.1881542193072303,3.141995562034325)}] plot[domain=2.418130807054263:3.8075164224901368,variable=\t]({1.*2.1722667712453996*cos(\t r)+0.*2.1722667712453996*sin(\t r)},{0.*2.1722667712453996*cos(\t r)+1.*2.1722667712453996*sin(\t r)});
\draw[line width=2pt] [shift={(0.918705883696099,3.531651337826396)}] plot[domain=2.778399084417675:4.713983402948865,variable=\t]({1.*0.8116523695107318*cos(\t r)+0.*0.8116523695107318*sin(\t r)},{0.*0.8116523695107318*cos(\t r)+1.*0.8116523695107318*sin(\t r)});
\draw[line width=2pt] [shift={(0.5096774102310706,3.53)}] plot[domain=-1.101905668217869:1.1019056682178683,variable=\t]({1.*0.9080003456357711*cos(\t r)+0.*0.9080003456357711*sin(\t r)},{0.*0.9080003456357711*cos(\t r)+1.*0.9080003456357711*sin(\t r)});
\draw[line width=2pt] [shift={(2.0169398610202194,3.078554783445665)}] plot[domain=3.457513367713698:5.611078716593286,variable=\t]({1.*1.1540531146471664*cos(\t r)+0.*1.1540531146471664*sin(\t r)},{0.*1.1540531146471664*cos(\t r)+1.*1.1540531146471664*sin(\t r)});
\draw (-2.29955143165361,3.636230375882849) node[anchor=north west] {$X$};
\draw (0.5941778815198596,1.6129724821192877) node[anchor=north west] {$Y$};
\draw (0.5,4.012650449141186) node[anchor=north west] {$Z$};
\draw (2.029279410817271,4.412596776978169) node[anchor=north west] {$W$};
\begin{scriptsize}
\draw [fill=qqqqff] (-0.42,5.22) circle (2.5pt);
\draw [fill=qqqqff] (-0.44,4.58) circle (2.5pt);
\draw [fill=qqqqff] (0.16,3.82) circle (2.5pt);
\draw [fill=qqqqff] (0.92,4.34) circle (2.5pt);
\draw [fill=qqqqff] (0.98,4.98) circle (2.5pt);
\draw [fill=qqqqff] (-2.68,1.44) circle (2.5pt);
\draw [fill=qqqqff] (0.34,-0.18) circle (2.5pt);
\draw [fill=qqqqff] (2.92,2.36) circle (2.5pt);
\draw [fill=qqqqff] (-0.52,1.8) circle (2.5pt);
\draw [fill=qqqqff] (0.92,2.72) circle (2.5pt);
\end{scriptsize}
\end{tikzpicture}\label{Fig_02_a_2.tikz}}
  \caption{Diagrams used in the proof of
    Proposition~\ref{prop:not_two_two}, Step~6, case (i).}
\end{figure}

\begin{figure}
  \centering
  \begin{tikzpicture}[line cap=round,line join=round,>=triangle 45,x=0.8cm,y=0.8cm]
\clip(-4.16,0.6) rectangle (10.6,6.26);
\draw [shift={(-0.6230642496536478,4.79)},line width=2.pt]  plot[domain=1.0471975511965987:2.0943951023931957,variable=\t]({1.*1.073871500692704*cos(\t r)+0.*1.073871500692704*sin(\t r)},{0.*1.073871500692704*cos(\t r)+1.*1.073871500692704*sin(\t r)});
\draw [line width=2.pt] (-1.16,5.72)-- (-1.16,5.1);
\draw [line width=2.pt] (-1.16,5.1)-- (-0.6230642496536478,4.79);
\draw [line width=2.pt] (-0.6230642496536478,4.79)-- (-0.08612849930729616,5.1);
\draw [line width=2.pt] (-0.08612849930729616,5.1)-- (-0.08612849930729688,5.72);
\draw [shift={(-2.2840406529103787,4.47348167539267)},line width=2.pt]  plot[domain=0.8370184212254427:3.86693179067324,variable=\t]({1.*1.6784741055426067*cos(\t r)+0.*1.6784741055426067*sin(\t r)},{0.*1.6784741055426067*cos(\t r)+1.*1.6784741055426067*sin(\t r)});
\draw [shift={(-1.9082278481012658,2.6588607594936704)},line width=2.pt]  plot[domain=2.735765011175732:5.357728925073849,variable=\t]({1.*1.7760283191126258*cos(\t r)+0.*1.7760283191126258*sin(\t r)},{0.*1.7760283191126258*cos(\t r)+1.*1.7760283191126258*sin(\t r)});
\draw [shift={(0.3247632508833923,2.9569611307420494)},line width=2.pt]  plot[domain=-2.1668688331381425:0.19550105591720343,variable=\t]({1.*2.074759975295327*cos(\t r)+0.*2.074759975295327*sin(\t r)},{0.*2.074759975295327*cos(\t r)+1.*2.074759975295327*sin(\t r)});
\draw [shift={(0.9225363739973566,4.317776091219923)},line width=2.pt]  plot[domain=-0.5877458609449349:2.1943794655589834,variable=\t]({1.*1.7273206757846618*cos(\t r)+0.*1.7273206757846618*sin(\t r)},{0.*1.7273206757846618*cos(\t r)+1.*1.7273206757846618*sin(\t r)});
\draw [line width=2.pt] (-3.54,3.36)-- (-2.,4.);
\draw [line width=2.pt] (-2.,4.)-- (-0.62,3.28);
\draw [line width=2.pt] (-0.62,3.28)-- (-0.84,1.24);
\draw [shift={(6.916935750346353,4.71)},line width=2.pt]  plot[domain=1.0471975511965987:2.0943951023931957,variable=\t]({1.*1.073871500692704*cos(\t r)+0.*1.073871500692704*sin(\t r)},{0.*1.073871500692704*cos(\t r)+1.*1.073871500692704*sin(\t r)});
\draw [shift={(5.25595934708962,4.39348167539267)},line width=2.pt]  plot[domain=0.8370184212254421:3.8669317906732403,variable=\t]({1.*1.6784741055426076*cos(\t r)+0.*1.6784741055426076*sin(\t r)},{0.*1.6784741055426076*cos(\t r)+1.*1.6784741055426076*sin(\t r)});
\draw [shift={(5.631772151898734,2.5788607594936708)},line width=2.pt]  plot[domain=2.735765011175732:5.357728925073849,variable=\t]({1.*1.7760283191126254*cos(\t r)+0.*1.7760283191126254*sin(\t r)},{0.*1.7760283191126254*cos(\t r)+1.*1.7760283191126254*sin(\t r)});
\draw [shift={(7.8647632508833905,2.8769611307420493)},line width=2.pt]  plot[domain=-2.1668688331381407:0.1955010559172032,variable=\t]({1.*2.0747599752953256*cos(\t r)+0.*2.0747599752953256*sin(\t r)},{0.*2.0747599752953256*cos(\t r)+1.*2.0747599752953256*sin(\t r)});
\draw [shift={(8.462536373997356,4.237776091219921)},line width=2.pt]  plot[domain=-0.587745860944934:2.1943794655589826,variable=\t]({1.*1.7273206757846618*cos(\t r)+0.*1.7273206757846618*sin(\t r)},{0.*1.7273206757846618*cos(\t r)+1.*1.7273206757846618*sin(\t r)});
\draw [line width=2.pt] (6.38,5.64)-- (6.38,5.02);
\draw [line width=2.pt] (6.38,5.02)-- (6.916935750346353,4.71);
\draw [line width=2.pt] (6.916935750346353,4.71)-- (7.453871500692704,5.02);
\draw [line width=2.pt] (7.453871500692704,5.02)-- (7.453871500692703,5.64);
\draw [line width=2.pt] (4.,3.28)-- (5.54,3.92);
\draw [line width=2.pt] (5.54,3.92)-- (6.92,3.2);
\draw [line width=2.pt] (6.92,3.2)-- (6.7,1.16);
\draw [line width=2.pt] (6.38,5.02)-- (5.54,3.92);
\draw [line width=2.pt] (6.916935750346353,4.71)-- (6.92,3.2);
\draw [line width=2.pt] (7.453871500692704,5.02)-- (9.9,3.28);
\draw (-1.82-0.08,5.52) node[anchor=north west] {${x_1}$};
\draw (-1.02,4.88-0.12) node[anchor=north west] {${x_2}$};
\draw (0.,5.52) node[anchor=north west] {${x_3}$};     
\draw (5.62,5.52) node[anchor=north west] {${x_1}$};
\draw (6.84+0.08,4.84) node[anchor=north west] {${x_2}$};
\draw (7.58-0.08,5.48) node[anchor=north west] {${x_3}$};
\draw (2.52,3.74-0.10) node[anchor=north west] {${v}$};
\draw (10.06,3.74-0.12) node[anchor=north west] {${v}$};
\draw (-2.72,4.68-0.12) node[anchor=north west] {${y_1}$};
\draw (-0.56,3.7) node[anchor=north west] {${y_2}$};
\draw (4.74,4.58-0.10) node[anchor=north west] {${y_1}$};
\draw (6.96+0.06,3.5) node[anchor=north west] {${y_2}$};
\draw (-1.04,6.02-0.28) node[anchor=north west] {${C_2}$};
\draw (6.46,5.92-0.22) node[anchor=north west] {${C_2}$};
\draw (-2.62,3.02) node[anchor=north west] {${E'_2}$};
\draw (5.1,2.94) node[anchor=north west] {${E'_2}$};
\draw (4.5,5.62) node[anchor=north west] {${X}$};
\draw (6.+0.08,4.76-0.18) node[anchor=north west] {${Y}$};
\draw (8.34,3.12) node[anchor=north west] {${Z}$};
\draw (8.6,5.54) node[anchor=north west] {${W}$};
\draw [fill=qqqqff] (-1.16,5.72) circle (2.5pt);
\draw [fill=qqqqff] (-1.16,5.1) circle (2.5pt);
\draw [fill=qqqqff] (-0.6230642496536478,4.79) circle (2.5pt);
\draw [fill=qqqqff] (-0.08612849930729616,5.1) circle (2.5pt);
\draw [fill=qqqqff] (-0.08612849930729688,5.72) circle (2.5pt);
\draw [fill=qqqqff] (-2.,4.) circle (2.5pt);
\draw [fill=qqqqff] (-0.62,3.28) circle (2.5pt);
\draw [fill=qqqqff] (-3.54,3.36) circle (2.5pt);
\draw [fill=qqqqff] (-0.84,1.24) circle (2.5pt);
\draw [fill=qqqqff] (2.36,3.36) circle (2.5pt);
\draw [fill=qqqqff] (6.38,5.64) circle (2.5pt);
\draw [fill=qqqqff] (6.38,5.64) circle (2.5pt);
\draw [fill=qqqqff] (4.,3.28) circle (2.5pt);
\draw [fill=qqqqff] (4.,3.28) circle (2.5pt);
\draw [fill=qqqqff] (6.7,1.16) circle (2.5pt);
\draw [fill=qqqqff] (6.7,1.16) circle (2.5pt);
\draw [fill=qqqqff] (9.9,3.28) circle (2.5pt);
\draw [fill=qqqqff] (9.9,3.28) circle (2.5pt);
\draw [fill=qqqqff] (7.453871500692703,5.64) circle (2.5pt);
\draw [fill=qqqqff] (6.38,5.64) circle (2.5pt);
\draw [fill=qqqqff] (6.38,5.02) circle (2.5pt);
\draw [fill=qqqqff] (6.38,5.02) circle (2.5pt);
\draw [fill=qqqqff] (6.916935750346353,4.71) circle (2.5pt);
\draw [fill=qqqqff] (6.916935750346353,4.71) circle (2.5pt);
\draw [fill=qqqqff] (7.453871500692704,5.02) circle (2.5pt);
\draw [fill=qqqqff] (7.453871500692704,5.02) circle (2.5pt);
\draw [fill=qqqqff] (7.453871500692703,5.64) circle (2.5pt);
\draw [fill=qqqqff] (4.,3.28) circle (2.5pt);
\draw [fill=qqqqff] (5.54,3.92) circle (2.5pt);
\draw [fill=qqqqff] (5.54,3.92) circle (2.5pt);
\draw [fill=qqqqff] (6.92,3.2) circle (2.5pt);
\draw [fill=qqqqff] (6.92,3.2) circle (2.5pt);
\draw [fill=qqqqff] (6.7,1.16) circle (2.5pt);
\end{tikzpicture}

  \caption{Diagram used in the proof of
    Proposition~\ref{prop:not_two_two}, Step~6, case (ii).}
  \label{Figura4.tikz}
\end{figure}
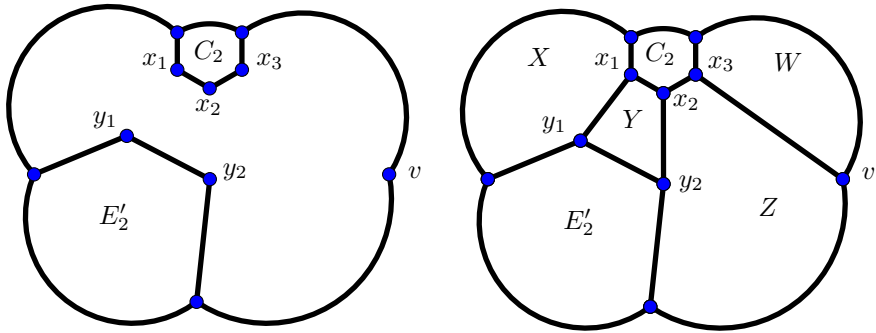

\emph{Step 6}: we claim that $E_0$ has not five edges. Suppose by
contradiction that $E_0$ has exactly five edges and consider two possible
cases: $E_2'$ has either (i) three or (ii) four edges.

If $E_2'$ has three edges the region $E_2 = C_2 \cup E_2'$ has 8
distinct vertices (since $C_2$ has five vertices). Three vertices of
$C_2$ (let us call them $x_1$, $x_2$ and $x_3$) are not vertices of
$E_0$, and one vertex of $E_2'$ (let us call it $y$) is not a vertex
of $E_0$. On the other hand $E_0$ has five vertices, and four of them
are shared by $C_2$ and $E_2'$. We denote by $v$ the remaining
vertex. Up to now we have considered 9 vertices in total, since the
cluster $\E$ has exactly 10 vertices, there is an additional vertex
$w$ belonging to neither $E_0$ nor $E_2$.
The situation is depicted in Figure~\ref{Fig_02_a.tikz}.
We see that 11 edges have been
already identified, so 4 edges are missing.

Consider the three edges
which meet in the vertex $w$. At least two of them should connect $w$
to the vertices $x_k$ of $C_2$. In fact if only one edges connects $w$
to $C_2$ the other two edges of $w$ should go to $v$ and $y$ and hence
the two remaining vertices of $C_2$ should be joined together which
is not admitted (we would obtain a two sided component). Not all three
edges of $w$ can join the three free vertices of $C_2$ because
otherwise we would obtain two three-sided internal components. But we
know that at most one component can be internal. So, exactly two edges
join $w$ with two vertices of $C_2$. The two vertices of $C_2$ must be
consecutive, otherwise the third vertex $x_2$ could not be connected to
anything (the edge would be closed in the loop: $w$, $x_3$, $x_2$
$x_1$). We have two possibilities: the two vertices are either $x_1$
and $x_2$ or $x_2$ and $x_3$ (the order of the vertices is given by
the Figure, where $x_1$ is ``closer'' to the component $E_2'$).

In the first case ($x_1$ and $x_2$ are joined to $w$) the third edge
in $w$ cannot go to $x_3$ (already excluded) and cannot go to $v$
because otherwise the edge from $x_3$ to $y$ would cross the already
defined edges. So the diagram is completed by an edge joining $w$ with
$y$ and an edge joining $v$ with $x_3$.
The resulting diagram is depicted in Figure~\ref{Fig_02_a_1.tikz}.
We know that $C_1$ is external and has four or five edges: the only
possibility is $X=C_1$. Then $E_1'$ must be adjacent to $E_2'$ so it
must be $Z=E_1'$: however
$E_1'$ cannot be adjacent to $C_1$ and we get a contradiction.

In the second case ($x_2$ and $x_3$ are joined to $w$) we can complete
the diagram in a unique way, by adding an edge from $w$ to $v$ and an
edge from $y$ to $x_1$ as represented in Figure \ref{Fig_02_a_2.tikz}.
 In this case we have $X=E_1'$ since $E_1'$
must be adjacent to $E_2'$ but cannot have six edges. So $W=C_1$
because $C_1$ is external and not adjacent to $E_1'$. So $Y$ and $Z$
are the two connected regions $E_3$ and $E_4$. However in \emph{Step 4}  we
proved that the connected region, if internal, must be adjacent to
both $E_1'$ and $E_2'$ which is not the case for the component
$Z$. So this configuration must be excluded, too.

So, the case when $E_2'$ has only three edges has been completed and
excluded. Suppose now (ii) that $E_2'$ has four edges. In this case no
additional vertex must be added, and we are in the situation depicted
in Figure \ref{Figura4.tikz}. Let $x_1$, $x_2$ and $x_3$ be the free vertices of
$C_2$ and $v$ be the free vertex of $E_0$, as before. Let $y_1$ and
$y_2$ be the two free vertices of $E_2'$. There are three edges
missing in the diagram and there is only one possibility (since the
edges from $C_2$ cannot go back to $C_2$ and they cannot cross each
other): $x_1$ is joined to $y_1$, $x_2$ to $y_2$ and $x_3$ to $v$.
The component $C_1$ is external with four or five edges, hence $C_1$ is
either $X$ or $Z$. The component $E_1'$ is adjacent to $E_2'$ but
cannot be adjacent to $C_1$ hence $E_1'$ is either $X$ or $Z$. So $Y$
and $W$ are the two connected regions $E_3$ and $E_4$: say $Y=E_3$ and
$W=E_4$.

But now we notice that $E_3$ is internal and $E_4$ is external with only three
edges, hence by Step~5 they should be adjacent, which is not the case.

\emph{Step 7}: conclusion. We know now that
$E_0$ has six edges. Recall that $C_2$ is external and has five
vertices, two of which are shared with $E_0$
while $E_2'$ has at least three vertices (all distinct from $C_2$)
two of which are shared with the vertices of $E_0$. So we
have identified 6 vertices of $E_0$ and at least $3+1=4$ internal
vertices of $E_2=C_2 \cup E_2'$. We know that the cluster has 10
vertices in total, so we have identified all of them. In particular we
conclude that $E_2'$ has three vertices. Let $x_1$, $x_2$ and $x_3$ be
the three internal vertices of $C_2$ and let $v$ be the internal
vertex of $E_2'$.

If we look at the edges, we have already identified the six edges of
$E_0$, other four are the internal edges of $C_2$ and other two are
the internal edges of $E_2'$. To reach the total of 15 edges, we need
to place other three edges. No edge can join two points of $C_2$
(otherwise a two sided component would rise). So the three missing
edges start from the three internal points of $C_2$. One of them goes
to the internal vertex of $E_2'$ and the other two go to the two free
vertices of $E_0$.

There are now two
possibilities: either (a) the vertex $v$ is connected to the middle of
the three internal vertices of $C_2$ or (b) it is connected to one
lateral vertex (see Figure \ref{fig:ezero6lati})

\begin{figure}
  \centering
  \subfigure[]{
  \begin{tikzpicture}[line cap=round,line join=round,>=triangle
    45,x=1.0cm,y=1.0cm]
  \clip(1.7,0.28) rectangle (6.5,6.00);
  \draw[line width=2pt] (3.1,4.7)-- (3.12,3.78);
  \draw[line width=2pt] (3.12,3.78)-- (4.1,3.14);
  \draw[line width=2pt] (4.1,3.14)-- (5.,4.);
  \draw[line width=2pt] (5.,4.)-- (5.,5.);
  \draw[line width=2pt] [shift={(4.12132075471698,4.398301886792452)}]
  plot[domain=0.6004350063641383:2.854361401189685,variable=\t]({1.*1.0649496868531663*cos(\t
    r)+0.*1.0649496868531663*sin(\t r)},{0.*1.0649496868531663*cos(\t
    r)+1.*1.0649496868531663*sin(\t r)});
  \draw[line width=2pt] (4.1,3.14)-- (4.1,1.82);
  \draw[line width=2pt] (4.1,1.82)-- (3.003025321539324,1.18);
  \draw[line width=2pt] (4.1,1.82)-- (5.123025321539324,1.02);
  \draw[line width=2pt] (3.12,3.78)-- (2.28,3.4);
  \draw[line width=2pt] (5.,4.)-- (6.1,3.68);
  \draw[line width=2pt] [shift={(2.741169134840218,4.017724084177709)}]
  plot[domain=1.0866188360685638:4.071092025588856,variable=\t]({1.*0.7708826208330711*cos(\t
    r)+0.*0.7708826208330711*sin(\t r)},{0.*0.7708826208330711*cos(\t
    r)+1.*0.7708826208330711*sin(\t r)});
  \draw[line width=2pt] [shift={(2.9233057310147785,2.3817763627124)}]
  plot[domain=2.134265002914584:4.77862673523436,variable=\t]({1.*1.204417551801554*cos(\t
    r)+0.*1.204417551801554*sin(\t r)},{0.*1.204417551801554*cos(\t
    r)+1.*1.204417551801554*sin(\t r)});
  \draw[line width=2pt] [shift={(4.099629511175488,1.585005512679163)}]
  plot[domain=3.495380464155704:5.778739714956995,variable=\t]({1.*1.1690039409805677*cos(\t
    r)+0.*1.1690039409805677*sin(\t r)},{0.*1.1690039409805677*cos(\t
    r)+1.*1.1690039409805677*sin(\t r)});
  \draw[line width=2pt] [shift={(4.627250425997023,2.7115034888488663)}]
  plot[domain=-1.2856843678161134:0.5817071567969192,variable=\t]({1.*1.7626618506786451*cos(\t
    r)+0.*1.7626618506786451*sin(\t r)},{0.*1.7626618506786451*cos(\t
    r)+1.*1.7626618506786451*sin(\t r)});
  \draw[line width=2pt] [shift={(5.461740614334471,4.2664505119453935)}]
  plot[domain=-0.7431203943952518:2.132596946788658,variable=\t]({1.*0.8667751994324433*cos(\t
    r)+0.*0.8667751994324433*sin(\t r)},{0.*0.8667751994324433*cos(\t
    r)+1.*0.8667751994324433*sin(\t r)});
  \draw (3.60,4.505823063679203) node[anchor=north west]
        {$C_2$};
  \draw (3.87780417066169,6.05) node[anchor=north west] {$E_0$};
          \draw [fill=qqqqff] (3.1,4.7) circle (2.5pt);
          \draw [fill=qqqqff] (3.12,3.78) circle (2.5pt);
          \draw [fill=qqqqff] (4.1,3.14) circle (2.5pt);
          \draw [fill=qqqqff] (5.,4.) circle (2.5pt);
          \draw [fill=qqqqff] (5.,5.) circle (2.5pt);
          \draw [fill=qqqqff] (4.1,1.82) circle (2.5pt);
          \draw (4.4,2.076860436363506) node {$v$};
          \draw (4.3,1.0) node {$E_2'$};
          \draw [fill=qqqqff] (3.003025321539324,1.18) circle (2.5pt);
          \draw [fill=qqqqff] (5.123025321539324,1.02) circle (2.5pt);
          \draw [fill=qqqqff] (2.28,3.4) circle (2.5pt);
          \draw [fill=qqqqff] (6.1,3.68) circle (2.5pt);
\end{tikzpicture}
}
  \subfigure[]{
  \begin{tikzpicture}[line cap=round,line join=round,>=triangle
    45,x=1.0cm,y=1.0cm]
  \clip(8.30,2.0) rectangle (14.02,6.23);
  \draw[line width=2pt] (9.816642462132213,5.144989859727443)--
  (9.798155122901875,4.16516088051952);
  \draw[line width=2pt] (9.798155122901875,4.16516088051952)--
  (10.716390827937389,3.5849298560903002);
  \draw[line width=2pt] (10.716390827937389,3.5849298560903002)--
  (11.609914367475016,4.276084915901548);
  \draw[line width=2pt] (11.609914367475016,4.276084915901548)--
  (11.609914367475016,5.348350591261163);
  \draw[line width=2pt] (11.609914367475016,4.276084915901548)--
  (12.514662930507589,3.7002037088191595);
  \draw[line width=2pt] (12.514662930507589,3.7002037088191595)--
  (13.506018063975777,4.138244349188824);
  \draw[line width=2pt] (12.514662930507589,3.7002037088191595)--
  (12.606882012690676,2.8471771986256016);
  \draw[line width=2pt] (10.716390827937389,3.5849298560903002)--
  (10.670281286845844,2.616629493167883);
  \draw[line width=2pt] (9.798155122901875,4.16516088051952)--
  (8.825899643184101,3.7002037088191595);
  \draw[line width=2pt] [shift={(10.76730623933436,4.7702430455413785)}]
  plot[domain=0.6013313078618941:2.7661006672331054,variable=\t]({1.*1.0218594776322485*cos(\t
    r)+0.*1.0218594776322485*sin(\t r)},{0.*1.0218594776322485*cos(\t
    r)+1.*1.0218594776322485*sin(\t r)});
  \draw[line width=2pt] [shift={(9.378304647510683,4.383486757849381)}]
  plot[domain=1.048494845315568:4.0325125261776105,variable=\t]({1.*0.8786506779699703*cos(\t
    r)+0.*0.8786506779699703*sin(\t r)},{0.*0.8786506779699703*cos(\t
    r)+1.*0.8786506779699703*sin(\t r)});
  \draw[line width=2pt] [shift={(9.848784702593882,3.329811047936344)}]
  plot[domain=2.7941740840891423:5.568248638228706,variable=\t]({1.*1.0878807692020311*cos(\t
    r)+0.*1.0878807692020311*sin(\t r)},{0.*1.0878807692020311*cos(\t
    r)+1.*1.0878807692020311*sin(\t r)});
  \draw[line width=2pt] [shift={(11.59468783479973,3.100611391632405)}]
  plot[domain=3.6239097934821523:6.037848085603502,variable=\t]({1.*1.0434394778526175*cos(\t
    r)+0.*1.0434394778526175*sin(\t r)},{0.*1.0434394778526175*cos(\t
    r)+1.*1.0434394778526175*sin(\t r)});
  \draw[line width=2pt] [shift={(13.027147165927188,3.513118131475707)}]
  plot[domain=-2.133759216824422:0.9171091995872755,variable=\t]({1.*0.7874643643176354*cos(\t
    r)+0.*0.7874643643176354*sin(\t r)},{0.*0.7874643643176354*cos(\t
    r)+1.*0.7874643643176354*sin(\t r)});
  \draw[line width=2pt] [shift={(12.52507059365559,4.691753640074645)}]
  plot[domain=-0.5137253285695706:2.519237962492979,variable=\t]({1.*1.1263349744301199*cos(\t
    r)+0.*1.1263349744301199*sin(\t r)},{0.*1.1263349744301199*cos(\t
    r)+1.*1.1263349744301199*sin(\t r)});
  \draw (10.514570208435694,6.27) node[anchor=north west]
        {$E_0$};
  \draw (10.494674831805597,4.862192192007618) node[anchor=north west] {$C_2$};
  \draw [fill=qqqqff] (9.816642462132213,5.144989859727443)
          circle (2.5pt);
          \draw [fill=qqqqff] (9.798155122901875,4.16516088051952)
          circle (2.5pt);
          \draw [fill=qqqqff] (10.716390827937389,3.5849298560903002)
          circle (2.5pt);
          \draw [fill=qqqqff] (11.609914367475016,4.276084915901548)
          circle (2.5pt);
          \draw [fill=qqqqff] (11.609914367475016,5.348350591261163)
          circle (2.5pt);
          \draw [fill=qqqqff] (12.514662930507589,3.7002037088191595)
          circle (2.5pt);
          \draw (12.583689377965822,3.98) node {$v$};
          \draw (13.28,3.20) node {$E_2'$};
          \draw [fill=qqqqff] (13.506018063975777,4.138244349188824) circle (2.5pt);
          \draw [fill=qqqqff] (12.606882012690676,2.8471771986256016)
          circle (2.5pt);
          \draw [fill=qqqqff] (10.670281286845844,2.616629493167883)
          circle (2.5pt);
          \draw [fill=qqqqff] (8.825899643184101,3.7002037088191595)
          circle (2.5pt);
\end{tikzpicture}
}
  \caption{Diagrams used in the proof of
    Proposition~\ref{prop:not_two_two} Step~7.}
  \label{fig:ezero6lati}
\end{figure}
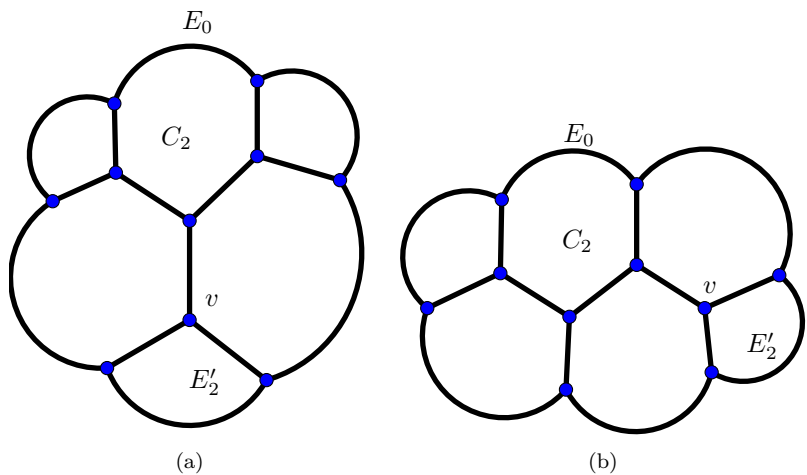

We can easily exclude case (a) because the component $C_1$ must be one
of the two five-sided components ($C_1$ has either four or five edges
and there are no components with four edges) while $E_1'$ must be
adjacent to $E_2'$ and hence must be the other component with five
edges. But this is a contradiction since $C_1$ cannot be adjacent to $E_1'$.

So we remain with the configuration of case (b). The
region with three edges adjacent to $C_2$ is not $C_1$ (because $C_1$
has four or five
edges) and it cannot be $E_1'$ because $E_1'$ must be adjacent to
$E_2'$. Hence we conclude that it is one of $E_3$ and $E_4$.
Let us say it is $E_3$. Then $E_4$ must be the
region with five edges, because otherwise $C_1$ and $E_1'$ would be
adjacent to each other. So $C_1$ has four edges and hence
$p_1\ge k_0/4$ is the region with higher pressure and $p_2\ge k_0/8$ is the
second higher pressure while $p_3+p_4 \le k_0/8$.

We know that $E_3$ has three edges, $E_4$ has five edges
and both $E_3$ and $E_4$ are external. Let $L_{j,k}$ be the total
length of the edges between $E_j$ and $E_k$. Applying Proposition \ref{prop:varI}
we obtain, for $j=3,4$:
\begin{equation}\label{eq:Lj0}
  L_{j,0} \le 2 \sqrt{\pi} \sqrt{m(E_j)} = 2 \sqrt{\pi}
\end{equation}
Since $p_1$ and $p_2$ are the
largest pressures and $E_3$ is not adjacent to $E_4$
we have,
for $j=3,4$
\[
L_{j,0}\, p_j \ge \sum_{k=1}^4 L_{j,k} (p_j - p_k)
\]
hence, by Lemma~\ref{lem:turning_angle}
\begin{equation}\label{eq:pressioni34}
L_{3,0}\, p_3 \ge \pi, \qquad
L_{4,0}\, p_4 \ge \frac{\pi}{3}
\end{equation}
and putting together with \eqref{eq:Lj0} we obtain
\[
p_3 \ge \frac{\pi}{L_{3,0}} \ge  \frac{\sqrt \pi}{2}, \qquad
p_4 \ge \frac{\pi}{3L_{4,0}} \ge  \frac{\sqrt \pi}{6}.
\]

Now we are going to improve the estimates on $p_1$ and $p_2$.
First notice that if we denote by $\ell_i$ the length of the external
edge of $C_i$ we have, by Proposition~\ref{prop:varII} (notice that $m(C_i)<k_3<1$),
\[
\ell_i \le \frac{m(C_i)}{\lvert 2-m(C_i)\rvert} P(\E)
\le \frac{m(C_i)}{2-k_3}k_0
\]
while, by the isoperimetric inequality, we have
\[
 P(C_i) \ge 2 \sqrt{\pi}\sqrt{m(C_i)}.
\]
Now, applying Proposition \ref{prop:pressure_perimeter} to the component $C_i$ with $i=1,2$,
which has $n_i = i+3$ edges,
we have
\begin{align*}
p_i & \ge \frac{(6-n_i) \pi}{3 P(C_i)}  + p_{\mathrm{min}} \left(1-\frac{\ell_i}{P(C_i)}\right) \\
& \ge \frac{(3-i)\pi}{3 k_9} + \frac{\sqrt \pi}{6}\left(1-
\frac{m(C_i)k_0}{(2-k_3) 2 \sqrt \pi \sqrt{m(C_i)}}\right)\\
& = \frac{(3-i)\pi}{3 k_9} + \frac{\sqrt \pi}{6}\left(1-
\frac{\sqrt{m(C_i)}k_0}{2\sqrt \pi (2-k_3)}\right)\\
& \ge\frac{(3-i)\pi}{3 k_9} + \frac{\sqrt \pi}{6}\left(1-
\frac{\sqrt{k_3}k_0}{2\sqrt \pi (2-k_3)}\right)
\ge \frac{(3-i)\pi}{3 k_9} + c_7
\end{align*}
with $c_7 \defeq 0.1992$. 
By using \eqref{eq:pressioni34}
\begin{align*}
P(\E) &= 2(p_1+p_2+p_3+p_4)
\ge 2 \left(
\frac{2\pi}{3k_9} + c_7
+ \frac{\pi}{3k_9} + c_7
+ \frac{\sqrt  \pi}{2}
+ \frac{\sqrt \pi}{6}
\right)\\
&=  \frac{2\pi}{k_9} + 4 c_7 + \frac{4}{3} \sqrt \pi
\geq 11.9428 > k_0
\end{align*}
which is a contradiction.

\end{proof}

\section{Clusters with five components}\label{sec:five}

In this section we consider a weak minimizer $\E \in \M^*(1,1,1,1)$
with five bound\-ed compo\-nents. Only one region is disconnected: we
will assume the region is $E_1$ and we denote with $E_1'$ and $C_1$
respectively, its big and small connected
components.


\begin{figure}
\centering
\begin{tikzpicture}[line cap=round,line join=round,>=triangle 45,x=0.6cm,y=0.6cm]
\clip(-1.25,-4.9) rectangle (18.2,6.8);
\draw [shift={(2,5.45)},line width=2pt]  plot[domain=0.39:2.75,variable=\t]({1*0.91*cos(\t r)+0*0.91*sin(\t r)},{0*0.91*cos(\t r)+1*0.91*sin(\t r)});
\draw [shift={(2,4.04)},line width=2pt]  plot[domain=3.44:5.99,variable=\t]({1*1.44*cos(\t r)+0*1.44*sin(\t r)},{0*1.44*cos(\t r)+1*1.44*sin(\t r)});
\draw [shift={(2,3.56)},line width=2pt]  plot[domain=3.39:6.03,variable=\t]({1*2.56*cos(\t r)+0*2.56*sin(\t r)},{0*2.56*cos(\t r)+1*2.56*sin(\t r)});
\draw [shift={(0.59,4.21)},line width=2pt]  plot[domain=1.23:4.02,variable=\t]({1*1.68*cos(\t r)+0*1.68*sin(\t r)},{0*1.68*cos(\t r)+1*1.68*sin(\t r)});
\draw [line width=2pt] (1.16,5.8)-- (1.4,5.12);
\draw [line width=2pt] (1.4,5.12)-- (2.6,5.12);
\draw [line width=2pt] (2.6,5.12)-- (2.84,5.8);
\draw [line width=2pt] (1.4,5.12)-- (0.62,3.62);
\draw [line width=2pt] (2.6,5.12)-- (3.38,3.62);
\draw [line width=2pt] (3.38,3.62)-- (4.48,2.92);
\draw [line width=2pt] (0.62,3.62)-- (-0.48,2.92);
\draw [shift={(3.41,4.21)},line width=2pt]  plot[domain=-0.88:1.91,variable=\t]({1*1.68*cos(\t r)+0*1.68*sin(\t r)},{0*1.68*cos(\t r)+1*1.68*sin(\t r)});
\draw [shift={(8.93,3.65)},line width=2pt]  plot[domain=3.42:4.34,variable=\t]({1*1.28*cos(\t r)+0*1.28*sin(\t r)},{0*1.28*cos(\t r)+1*1.28*sin(\t r)});
\draw [shift={(7.99,3.65)},line width=2pt]  plot[domain=5.09:6.01,variable=\t]({1*1.28*cos(\t r)+0*1.28*sin(\t r)},{0*1.28*cos(\t r)+1*1.28*sin(\t r)});
\draw [shift={(8.46,2.87)},line width=2pt]  plot[domain=1.25:1.89,variable=\t]({1*1.97*cos(\t r)+0*1.97*sin(\t r)},{0*1.97*cos(\t r)+1*1.97*sin(\t r)});
\draw [shift={(8.51,3.95)},line width=2pt]  plot[domain=2.27:3.82,variable=\t]({1*1.04*cos(\t r)+0*1.04*sin(\t r)},{0*1.04*cos(\t r)+1*1.04*sin(\t r)});
\draw [shift={(8.41,3.95)},line width=2pt]  plot[domain=-0.68:0.87,variable=\t]({1*1.04*cos(\t r)+0*1.04*sin(\t r)},{0*1.04*cos(\t r)+1*1.04*sin(\t r)});
\draw [line width=2pt] (8.46,2.46)-- (8.46,1.98);
\draw [line width=2pt] (7.84,4.74)-- (7.12,5.7);
\draw [line width=2pt] (9.08,4.74)-- (9.8,5.7);
\draw [shift={(8.46,4.81)},line width=2pt]  plot[domain=0.58:2.56,variable=\t]({1*1.61*cos(\t r)+0*1.61*sin(\t r)},{0*1.61*cos(\t r)+1*1.61*sin(\t r)});
\draw [shift={(7.96,3.9)},line width=2pt]  plot[domain=2.01:4.97,variable=\t]({1*1.99*cos(\t r)+0*1.99*sin(\t r)},{0*1.99*cos(\t r)+1*1.99*sin(\t r)});
\draw [shift={(8.96,3.9)},line width=2pt]  plot[domain=-1.83:1.13,variable=\t]({1*1.99*cos(\t r)+0*1.99*sin(\t r)},{0*1.99*cos(\t r)+1*1.99*sin(\t r)});
\draw [shift={(8.46,10.5)},line width=2pt]  plot[domain=4.61:4.82,variable=\t]({1*7.24*cos(\t r)+0*7.24*sin(\t r)},{0*7.24*cos(\t r)+1*7.24*sin(\t r)});
\draw [shift={(15.5,2.8)},line width=2pt]  plot[domain=4.1:5.33,variable=\t]({1*1.52*cos(\t r)+0*1.52*sin(\t r)},{0*1.52*cos(\t r)+1*1.52*sin(\t r)});
\draw [shift={(15.5,4.26)},line width=2pt]  plot[domain=4.4:5.02,variable=\t]({1*2.1*cos(\t r)+0*2.1*sin(\t r)},{0*2.1*cos(\t r)+1*2.1*sin(\t r)});
\draw [shift={(15.9,1.51)},line width=2pt]  plot[domain=2.52:3.1,variable=\t]({1*1.28*cos(\t r)+0*1.28*sin(\t r)},{0*1.28*cos(\t r)+1*1.28*sin(\t r)});
\draw [shift={(15.1,1.51)},line width=2pt]  plot[domain=0.04:0.62,variable=\t]({1*1.28*cos(\t r)+0*1.28*sin(\t r)},{0*1.28*cos(\t r)+1*1.28*sin(\t r)});
\draw [shift={(15.75,3.14)},line width=2pt]  plot[domain=1.77:3.92,variable=\t]({1*1.25*cos(\t r)+0*1.25*sin(\t r)},{0*1.25*cos(\t r)+1*1.25*sin(\t r)});
\draw [shift={(15.25,3.14)},line width=2pt]  plot[domain=-0.78:1.37,variable=\t]({1*1.25*cos(\t r)+0*1.25*sin(\t r)},{0*1.25*cos(\t r)+1*1.25*sin(\t r)});
\draw [line width=2pt] (15.5,5.12)-- (15.5,4.36);
\draw [shift={(16.19,7.58)},line width=2pt]  plot[domain=3.92:4.44,variable=\t]({1*2.56*cos(\t r)+0*2.56*sin(\t r)},{0*2.56*cos(\t r)+1*2.56*sin(\t r)});
\draw [shift={(15.5,5.33)},line width=2pt]  plot[domain=0.38:2.76,variable=\t]({1*1.21*cos(\t r)+0*1.21*sin(\t r)},{0*1.21*cos(\t r)+1*1.21*sin(\t r)});
\draw [shift={(14.81,7.58)},line width=2pt]  plot[domain=4.99:5.5,variable=\t]({1*2.56*cos(\t r)+0*2.56*sin(\t r)},{0*2.56*cos(\t r)+1*2.56*sin(\t r)});
\draw [shift={(15.12,3.71)},line width=2pt]  plot[domain=1.91:4.49,variable=\t]({1*2.2*cos(\t r)+0*2.2*sin(\t r)},{0*2.2*cos(\t r)+1*2.2*sin(\t r)});
\draw [shift={(15.88,3.71)},line width=2pt]  plot[domain=-1.34:1.23,variable=\t]({1*2.2*cos(\t r)+0*2.2*sin(\t r)},{0*2.2*cos(\t r)+1*2.2*sin(\t r)});
\draw [shift={(2,-0.55)},line width=2pt]  plot[domain=0.6:2.54,variable=\t]({1*0.87*cos(\t r)+0*0.87*sin(\t r)},{0*0.87*cos(\t r)+1*0.87*sin(\t r)});
\draw [shift={(5.09,5.95)},line width=2pt]  plot[domain=4.15:4.26,variable=\t]({1*7.12*cos(\t r)+0*7.12*sin(\t r)},{0*7.12*cos(\t r)+1*7.12*sin(\t r)});
\draw [shift={(-1.09,5.95)},line width=2pt]  plot[domain=5.16:5.28,variable=\t]({1*7.12*cos(\t r)+0*7.12*sin(\t r)},{0*7.12*cos(\t r)+1*7.12*sin(\t r)});
\draw [line width=2pt] (2,-0.46)-- (2,-1.14);
\draw [shift={(1.94,-2.16)},line width=2pt]  plot[domain=1.51:4.01,variable=\t]({1*1.02*cos(\t r)+0*1.02*sin(\t r)},{0*1.02*cos(\t r)+1*1.02*sin(\t r)});
\draw [shift={(2,-1.59)},line width=2pt]  plot[domain=4.22:5.2,variable=\t]({1*1.53*cos(\t r)+0*1.53*sin(\t r)},{0*1.53*cos(\t r)+1*1.53*sin(\t r)});
\draw [shift={(2.06,-2.16)},line width=2pt]  plot[domain=-0.87:1.63,variable=\t]({1*1.02*cos(\t r)+0*1.02*sin(\t r)},{0*1.02*cos(\t r)+1*1.02*sin(\t r)});
\draw [shift={(6.42,-5.02)},line width=2pt]  plot[domain=2.76:2.99,variable=\t]({1*5.54*cos(\t r)+0*5.54*sin(\t r)},{0*5.54*cos(\t r)+1*5.54*sin(\t r)});
\draw [shift={(-2.42,-5.02)},line width=2pt]  plot[domain=0.15:0.38,variable=\t]({1*5.54*cos(\t r)+0*5.54*sin(\t r)},{0*5.54*cos(\t r)+1*5.54*sin(\t r)});
\draw [shift={(2,-3.31)},line width=2pt]  plot[domain=3.83:5.59,variable=\t]({1*1.37*cos(\t r)+0*1.37*sin(\t r)},{0*1.37*cos(\t r)+1*1.37*sin(\t r)});
\draw [shift={(1.4,-2.14)},line width=2pt]  plot[domain=1.63:4.49,variable=\t]({1*2.09*cos(\t r)+0*2.09*sin(\t r)},{0*2.09*cos(\t r)+1*2.09*sin(\t r)});
\draw [shift={(2.6,-2.14)},line width=2pt]  plot[domain=-1.35:1.51,variable=\t]({1*2.09*cos(\t r)+0*2.09*sin(\t r)},{0*2.09*cos(\t r)+1*2.09*sin(\t r)});
\draw [shift={(7.41,-1.2)},line width=2pt]  plot[domain=1.78:4.51,variable=\t]({1*1.31*cos(\t r)+0*1.31*sin(\t r)},{0*1.31*cos(\t r)+1*1.31*sin(\t r)});
\draw [shift={(8.48,-1.23)},line width=2pt]  plot[domain=0.78:2.37,variable=\t]({1*1.88*cos(\t r)+0*1.88*sin(\t r)},{0*1.88*cos(\t r)+1*1.88*sin(\t r)});
\draw [shift={(8.94,-1.78)},line width=2pt]  plot[domain=3.51:5.39,variable=\t]({1*1.93*cos(\t r)+0*1.93*sin(\t r)},{0*1.93*cos(\t r)+1*1.93*sin(\t r)});
\draw [shift={(10.52,-2.93)},line width=2pt]  plot[domain=-2.36:1.08,variable=\t]({1*0.51*cos(\t r)+0*0.51*sin(\t r)},{0*0.51*cos(\t r)+1*0.51*sin(\t r)});
\draw [shift={(10.07,-1.28)},line width=2pt]  plot[domain=-1.05:1.75,variable=\t]({1*1.38*cos(\t r)+0*1.38*sin(\t r)},{0*1.38*cos(\t r)+1*1.38*sin(\t r)});
\draw [line width=2pt] (7.14,0.08)-- (7.72,-1.2);
\draw [line width=2pt] (7.72,-1.2)-- (7.14,-2.48);
\draw [line width=2pt] (7.72,-1.2)-- (9.24,-1.2);
\draw [line width=2pt] (9.24,-1.2)-- (9.82,0.08);
\draw [line width=2pt] (9.24,-1.2)-- (10.08,-2.56);
\draw [line width=2pt] (10.08,-2.56)-- (10.76,-2.48);
\draw [line width=2pt] (10.08,-2.56)-- (10.16,-3.28);
\draw [shift={(12.76,-1.2)},line width=2pt]  plot[domain=0.87:5.42,variable=\t]({1*1*cos(\t r)+0*1*sin(\t r)},{0*1*cos(\t r)+1*1*sin(\t r)});
\draw [shift={(16.1,-0.49)},line width=2pt]  plot[domain=3.64:4.7,variable=\t]({1*3.07*cos(\t r)+0*3.07*sin(\t r)},{0*3.07*cos(\t r)+1*3.07*sin(\t r)});
\draw [shift={(16.69,-2.86)},line width=2pt]  plot[domain=-2.3:0.68,variable=\t]({1*0.94*cos(\t r)+0*0.94*sin(\t r)},{0*0.94*cos(\t r)+1*0.94*sin(\t r)});
\draw [shift={(14.18,-1.3)},line width=2pt]  plot[domain=1.15:2.31,variable=\t]({1*1.16*cos(\t r)+0*1.16*sin(\t r)},{0*1.16*cos(\t r)+1*1.16*sin(\t r)});
\draw [shift={(15.65,-1.78)},line width=2pt]  plot[domain=-0.26:2.14,variable=\t]({1*1.83*cos(\t r)+0*1.83*sin(\t r)},{0*1.83*cos(\t r)+1*1.83*sin(\t r)});
\draw [line width=2pt] (13.4,-0.44)-- (14,-1.2);
\draw [line width=2pt] (14,-1.2)-- (13.4,-1.96);
\draw [line width=2pt] (14,-1.2)-- (14.62,-1.2);
\draw [line width=2pt] (14.62,-1.2)-- (14.66,-0.24);
\draw [line width=2pt] (14.62,-1.2)-- (16.48,-2.36);
\draw [line width=2pt] (16.48,-2.36)-- (17.42,-2.26);
\draw [line width=2pt] (16.48,-2.36)-- (16.06,-3.56);
\draw (1.56,0.36) node[anchor=north west] {${{C_1}}$};
\draw (8,3.24) node[anchor=north west] {${{C_1}}$};
\draw (10.12,-2.62) node[anchor=north west] {${{C_1}}$};
\draw (1.54,6.04) node[anchor=north west] {${{C_1}}$};
\draw (13.66,-0.32) node[anchor=north west] {${{C_1}}$};
\draw (15.12,2.12) node[anchor=north west] {${{C_1}}$};
\draw (15.04,6.58) node[anchor=north west] {${{E'_1}}$};
\draw (7.98,0.4) node[anchor=north west] {${{E'_1}}$};
\draw (16.36,-2.44) node[anchor=north west] {${{E'_1}}$};
\draw (1.5,-3.14) node[anchor=north west] {${{E'_1}}$};
\draw (1.44,2.52) node[anchor=north west] {${{E'_1}}$};
\draw (8.06,6.18) node[anchor=north west] {${{E'_1}}$};
\draw (14.48,-1.82) node[anchor=north west] {${{E_2}}$};
\draw (-0.38,5.06) node[anchor=north west] {${{E_2}}$};
\draw (6.38,-0.56) node[anchor=north west] {${{E_2}}$};
\draw (-0.24,-1.4) node[anchor=north west] {${{E_2}}$};
\draw (6.24,4.36) node[anchor=north west] {${{E_2}}$};
\draw (13.36,4.44) node[anchor=north west] {${{E_2}}$};
\draw (12.38,-0.56) node[anchor=north west] {${{E_3}}$};
\draw (1.52,4.56) node[anchor=north west] {${{E_3}}$};
\draw (8.2,-1.78) node[anchor=north west] {${{E_3}}$};
\draw (8.06,4.66) node[anchor=north west] {${{E_3}}$};
\draw (1.62,-1.48) node[anchor=north west] {${{E_3}}$};
\draw (15.08,3.84) node[anchor=north west] {${{E_3}}$};
\draw (9.68,4.38) node[anchor=north west] {${{E_4}}$};
\draw (3.56,5) node[anchor=north west] {${{E_4}}$};
\draw (16.68,4.42) node[anchor=north west] {${{E_4}}$};
\draw (15.66,-0.56) node[anchor=north west] {${{E_4}}$};
\draw (10.06,-0.6) node[anchor=north west] {${{E_4}}$};
\draw (3.46,-1.52) node[anchor=north west] {${{E_4}}$};
\draw (-1.44,6.44) node[anchor=north west] {${{(A)}}$};
\draw (5.86,6.44) node[anchor=north west] {${{(B)}}$};
\draw (13.02,6.54) node[anchor=north west] {${{(C)}}$};
\draw (-1.46,0.78) node[anchor=north west] {${{(C')}}$};
\draw (5.94,0.84) node[anchor=north west] {${{(D)}}$};
\draw (12.98,0.98) node[anchor=north west] {${{(D')}}$};
\begin{scriptsize}
\fill [color=qqqqff] (1.4,5.12) circle (2.5pt);
\fill [color=qqqqff] (1.16,5.8) circle (2.5pt);
\fill [color=qqqqff] (0.62,3.62) circle (2.5pt);
\fill [color=qqqqff] (-0.48,2.92) circle (2.5pt);
\fill [color=qqqqff] (2.6,5.12) circle (2.5pt);
\fill [color=qqqqff] (2.84,5.8) circle (2.5pt);
\fill [color=qqqqff] (3.38,3.62) circle (2.5pt);
\fill [color=qqqqff] (4.48,2.92) circle (2.5pt);
\fill [color=qqqqff] (4.48,2.92) circle (2.5pt);
\fill [color=qqqqff] (2.84,5.8) circle (2.5pt);
\fill [color=qqqqff] (7.84,4.74) circle (2.5pt);
\fill [color=qqqqff] (7.12,5.7) circle (2.5pt);
\fill [color=qqqqff] (7.7,3.3) circle (2.5pt);
\fill [color=qqqqff] (8.46,2.46) circle (2.5pt);
\fill [color=qqqqff] (8.46,1.98) circle (2.5pt);
\fill [color=qqqqff] (9.08,4.74) circle (2.5pt);
\fill [color=qqqqff] (9.8,5.7) circle (2.5pt);
\fill [color=qqqqff] (9.22,3.3) circle (2.5pt);
\fill [color=qqqqff] (9.22,3.3) circle (2.5pt);
\fill [color=qqqqff] (8.46,2.46) circle (2.5pt);
\fill [color=qqqqff] (9.08,4.74) circle (2.5pt);
\fill [color=qqqqff] (9.22,3.3) circle (2.5pt);
\fill [color=qqqqff] (8.46,6.42) circle (2.5pt);
\fill [color=qqqqff] (9.8,5.7) circle (2.5pt);
\fill [color=qqqqff] (8.46,1.98) circle (2.5pt);
\fill [color=qqqqff] (14.38,5.78) circle (2.5pt);
\fill [color=qqqqff] (15.5,5.12) circle (2.5pt);
\fill [color=qqqqff] (15.5,4.36) circle (2.5pt);
\fill [color=qqqqff] (14.86,2.26) circle (2.5pt);
\fill [color=qqqqff] (14.62,1.56) circle (2.5pt);
\fill [color=qqqqff] (16.62,5.78) circle (2.5pt);
\fill [color=qqqqff] (16.14,2.26) circle (2.5pt);
\fill [color=qqqqff] (16.38,1.56) circle (2.5pt);
\fill [color=qqqqff] (16.14,2.26) circle (2.5pt);
\fill [color=qqqqff] (16.38,1.56) circle (2.5pt);
\fill [color=qqqqff] (15.5,4.36) circle (2.5pt);
\fill [color=qqqqff] (16.14,2.26) circle (2.5pt);
\fill [color=qqqqff] (15.5,6.54) circle (2.5pt);
\fill [color=qqqqff] (16.62,5.78) circle (2.5pt);
\fill [color=qqqqff] (15.5,5.12) circle (2.5pt);
\fill [color=qqqqff] (16.62,5.78) circle (2.5pt);
\fill [color=qqqqff] (16.38,1.56) circle (2.5pt);
\fill [color=qqqqff] (2,-0.46) circle (2.5pt);
\fill [color=qqqqff] (2,-1.14) circle (2.5pt);
\fill [color=qqqqff] (1.28,-0.06) circle (2.5pt);
\fill [color=qqqqff] (1.28,-2.94) circle (2.5pt);
\fill [color=qqqqff] (0.94,-4.18) circle (2.5pt);
\fill [color=qqqqff] (2.72,-0.06) circle (2.5pt);
\fill [color=qqqqff] (2.72,-2.94) circle (2.5pt);
\fill [color=qqqqff] (3.06,-4.18) circle (2.5pt);
\fill [color=qqqqff] (2.72,-0.06) circle (2.5pt);
\fill [color=qqqqff] (2,-0.46) circle (2.5pt);
\fill [color=qqqqff] (2,-1.14) circle (2.5pt);
\fill [color=qqqqff] (2.72,-2.94) circle (2.5pt);
\fill [color=qqqqff] (2.72,-2.94) circle (2.5pt);
\fill [color=qqqqff] (3.06,-4.18) circle (2.5pt);
\fill [color=qqqqff] (2.72,-0.06) circle (2.5pt);
\fill [color=qqqqff] (3.06,-4.18) circle (2.5pt);
\fill [color=qqqqff] (7.72,-1.2) circle (2.5pt);
\fill [color=qqqqff] (7.14,-2.48) circle (2.5pt);
\fill [color=qqqqff] (9.24,-1.2) circle (2.5pt);
\fill [color=qqqqff] (7.14,0.08) circle (2.5pt);
\fill [color=qqqqff] (9.82,0.08) circle (2.5pt);
\fill [color=qqqqff] (10.08,-2.56) circle (2.5pt);
\fill [color=qqqqff] (10.76,-2.48) circle (2.5pt);
\fill [color=qqqqff] (10.16,-3.28) circle (2.5pt);
\fill [color=qqqqff] (14,-1.2) circle (2.5pt);
\fill [color=qqqqff] (13.4,-1.96) circle (2.5pt);
\fill [color=qqqqff] (13.4,-0.44) circle (2.5pt);
\fill [color=qqqqff] (14.62,-1.2) circle (2.5pt);
\fill [color=qqqqff] (14.66,-0.24) circle (2.5pt);
\fill [color=qqqqff] (16.48,-2.36) circle (2.5pt);
\fill [color=qqqqff] (17.42,-2.26) circle (2.5pt);
\fill [color=qqqqff] (16.06,-3.56) circle (2.5pt);
\end{scriptsize}
\end{tikzpicture}
\\
\begin{tikzpicture}[line cap=round,line join=round,>=triangle 45,x=0.5cm,y=0.5cm]
\clip(-1.20,-1.7) rectangle (21.5,6.7);
\draw [shift={(-0.38,3.78)},line width=2pt]  plot[domain=1.51:4.78,variable=\t]({1*0.6*cos(\t r)+0*0.6*sin(\t r)},{0*0.6*cos(\t r)+1*0.6*sin(\t r)});
\draw [shift={(2.86,4.4)},line width=2pt]  plot[domain=3.51:5.2,variable=\t]({1*3.43*cos(\t r)+0*3.43*sin(\t r)},{0*3.43*cos(\t r)+1*3.43*sin(\t r)});
\draw [shift={(5.28,2.3)},line width=2pt]  plot[domain=-2.28:0.96,variable=\t]({1*1.22*cos(\t r)+0*1.22*sin(\t r)},{0*1.22*cos(\t r)+1*1.22*sin(\t r)});
\draw [shift={(4.2,3.68)},line width=2pt]  plot[domain=-0.21:2.09,variable=\t]({1*1.82*cos(\t r)+0*1.82*sin(\t r)},{0*1.82*cos(\t r)+1*1.82*sin(\t r)});
\draw [shift={(1.6,4.34)},line width=2pt]  plot[domain=0.49:3.12,variable=\t]({1*1.94*cos(\t r)+0*1.94*sin(\t r)},{0*1.94*cos(\t r)+1*1.94*sin(\t r)});
\draw [line width=2pt] (-0.34,4.38)-- (0.1,3.78);
\draw [line width=2pt] (0.1,3.78)-- (-0.34,3.18);
\draw [line width=2pt] (0.1,3.78)-- (3.04,3.78);
\draw [line width=2pt] (3.04,3.78)-- (3.3,5.26);
\draw [line width=2pt] (3.04,3.78)-- (4.62,2.64);
\draw [line width=2pt] (4.62,2.64)-- (4.48,1.38);
\draw [line width=2pt] (4.62,2.64)-- (5.98,3.3);
\draw [shift={(10.44,5.12)},line width=2pt]  plot[domain=0.16:2.99,variable=\t]({1*1.4*cos(\t r)+0*1.4*sin(\t r)},{0*1.4*cos(\t r)+1*1.4*sin(\t r)});
\draw [shift={(10.12,3.01)},line width=2pt]  plot[domain=2:4.23,variable=\t]({1*2.56*cos(\t r)+0*2.56*sin(\t r)},{0*2.56*cos(\t r)+1*2.56*sin(\t r)});
\draw [shift={(10.44,1.33)},line width=2pt]  plot[domain=3.52:5.91,variable=\t]({1*1.61*cos(\t r)+0*1.61*sin(\t r)},{0*1.61*cos(\t r)+1*1.61*sin(\t r)});
\draw [shift={(10.76,3.01)},line width=2pt]  plot[domain=-1.09:1.14,variable=\t]({1*2.56*cos(\t r)+0*2.56*sin(\t r)},{0*2.56*cos(\t r)+1*2.56*sin(\t r)});
\draw [line width=2pt] (9.06,5.34)-- (10.44,4.6);
\draw [line width=2pt] (10.44,4.6)-- (11.82,5.34);
\draw [line width=2pt] (10.44,4.6)-- (10.44,3.34);
\draw [line width=2pt] (9.86,2.26)-- (8.94,0.74);
\draw [line width=2pt] (11.02,2.26)-- (11.94,0.74);
\draw [line width=2pt] (9.86,2.26)-- (11.02,2.26);
\draw [shift={(10.76,2.47)},line width=2pt]  plot[domain=1.92:3.37,variable=\t]({1*0.92*cos(\t r)+0*0.92*sin(\t r)},{0*0.92*cos(\t r)+1*0.92*sin(\t r)});
\draw [shift={(10.12,2.47)},line width=2pt]  plot[domain=-0.23:1.22,variable=\t]({1*0.92*cos(\t r)+0*0.92*sin(\t r)},{0*0.92*cos(\t r)+1*0.92*sin(\t r)});
\draw [shift={(18.12,5.66)},line width=2pt]  plot[domain=0.25:2.89,variable=\t]({1*0.64*cos(\t r)+0*0.64*sin(\t r)},{0*0.64*cos(\t r)+1*0.64*sin(\t r)});
\draw [shift={(17.99,2.69)},line width=2pt]  plot[domain=1.73:3.94,variable=\t]({1*3.17*cos(\t r)+0*3.17*sin(\t r)},{0*3.17*cos(\t r)+1*3.17*sin(\t r)});
\draw [shift={(18.12,0.94)},line width=2pt]  plot[domain=3.36:6.07,variable=\t]({1*2.4*cos(\t r)+0*2.4*sin(\t r)},{0*2.4*cos(\t r)+1*2.4*sin(\t r)});
\draw [shift={(18.25,2.69)},line width=2pt]  plot[domain=-0.8:1.41,variable=\t]({1*3.17*cos(\t r)+0*3.17*sin(\t r)},{0*3.17*cos(\t r)+1*3.17*sin(\t r)});
\draw [shift={(18.43,2.34)},line width=2pt]  plot[domain=1.78:3.38,variable=\t]({1*1.47*cos(\t r)+0*1.47*sin(\t r)},{0*1.47*cos(\t r)+1*1.47*sin(\t r)});
\draw [shift={(17.81,2.34)},line width=2pt]  plot[domain=-0.24:1.36,variable=\t]({1*1.47*cos(\t r)+0*1.47*sin(\t r)},{0*1.47*cos(\t r)+1*1.47*sin(\t r)});
\draw [shift={(18.12,3.8)},line width=2pt]  plot[domain=4.16:5.27,variable=\t]({1*2.12*cos(\t r)+0*2.12*sin(\t r)},{0*2.12*cos(\t r)+1*2.12*sin(\t r)});
\draw [line width=2pt] (17,2)-- (15.78,0.42);
\draw [line width=2pt] (19.24,2)-- (20.46,0.42);
\draw [line width=2pt] (17.5,5.82)-- (18.12,5.38);
\draw [line width=2pt] (18.12,5.38)-- (18.74,5.82);
\draw [line width=2pt] (18.12,5.38)-- (18.12,3.78);
\draw (17.47,6.36) node[anchor=north west] {${{C_1}}$};
\draw (-1.14,4.37) node[anchor=north west] {${{C_1}}$};
\draw (9.85,3.22) node[anchor=north west] {${{C_1}}$};
\draw (17.62,3.46) node[anchor=north west] {${{E'_1}}$};
\draw (4.9,2.74) node[anchor=north west] {${{E'_1}}$};
\draw (9.93,6.52) node[anchor=north west] {${{E'_1}}$};
\draw (15.75,3.89) node[anchor=north west] {${{E_2}}$};
\draw (7.89,3.92) node[anchor=north west] {${{E_2}}$};
\draw (1.71,3.06) node[anchor=north west] {${{E_2}}$};
\draw (19.86,3.86) node[anchor=north west] {${{E_3}}$};
\draw (11.66,3.98) node[anchor=north west] {${{E_3}}$};
\draw (4.12,4.76) node[anchor=north west] {${{E_3}}$};
\draw (17.65,1.64) node[anchor=north west] {${{E_4}}$};
\draw (9.99,1.76) node[anchor=north west] {${{E_4}}$};
\draw (1.01,5.86) node[anchor=north west] {${{E_4}}$};
\draw (-0.97,6.58) node[anchor=north west] {${{(E)}}$};
\draw (7.55,6.58) node[anchor=north west] {${{(F)}}$};
\draw (15.37,6.78) node[anchor=north west] {${{(F')}}$};
\begin{scriptsize}
\fill [color=qqqqff] (0.1,3.78) circle (2.5pt);
\fill [color=qqqqff] (-0.34,3.18) circle (2.5pt);
\fill [color=qqqqff] (3.04,3.78) circle (2.5pt);
\fill [color=qqqqff] (4.62,2.64) circle (2.5pt);
\fill [color=qqqqff] (-0.34,4.38) circle (2.5pt);
\fill [color=qqqqff] (5.98,3.3) circle (2.5pt);
\fill [color=qqqqff] (4.48,1.38) circle (2.5pt);
\fill [color=qqqqff] (3.3,5.26) circle (2.5pt);
\fill [color=qqqqff] (10.44,4.6) circle (2.5pt);
\fill [color=qqqqff] (9.06,5.34) circle (2.5pt);
\fill [color=qqqqff] (11.82,5.34) circle (2.5pt);
\fill [color=qqqqff] (10.44,3.34) circle (2.5pt);
\fill [color=qqqqff] (9.86,2.26) circle (2.5pt);
\fill [color=qqqqff] (11.02,2.26) circle (2.5pt);
\fill [color=qqqqff] (8.94,0.74) circle (2.5pt);
\fill [color=qqqqff] (11.94,0.74) circle (2.5pt);
\fill [color=qqqqff] (11.82,5.34) circle (2.5pt);
\fill [color=qqqqff] (11.94,0.74) circle (2.5pt);
\fill [color=qqqqff] (10.44,3.34) circle (2.5pt);
\fill [color=qqqqff] (11.02,2.26) circle (2.5pt);
\fill [color=qqqqff] (18.12,5.38) circle (2.5pt);
\fill [color=qqqqff] (17.5,5.82) circle (2.5pt);
\fill [color=qqqqff] (18.74,5.82) circle (2.5pt);
\fill [color=qqqqff] (18.12,3.78) circle (2.5pt);
\fill [color=qqqqff] (17,2) circle (2.5pt);
\fill [color=qqqqff] (19.24,2) circle (2.5pt);
\fill [color=qqqqff] (15.78,0.42) circle (2.5pt);
\fill [color=qqqqff] (20.46,0.42) circle (2.5pt);
\fill [color=qqqqff] (18.74,5.82) circle (2.5pt);
\fill [color=qqqqff] (20.46,0.42) circle (2.5pt);
\fill [color=qqqqff] (18.12,3.78) circle (2.5pt);
\fill [color=qqqqff] (19.24,2) circle (2.5pt);
\end{scriptsize}
\end{tikzpicture}
\caption{Classification of clusters with five components, Proposition~\ref{classify_five_components}.}
\label{classificazione}
\end{figure}
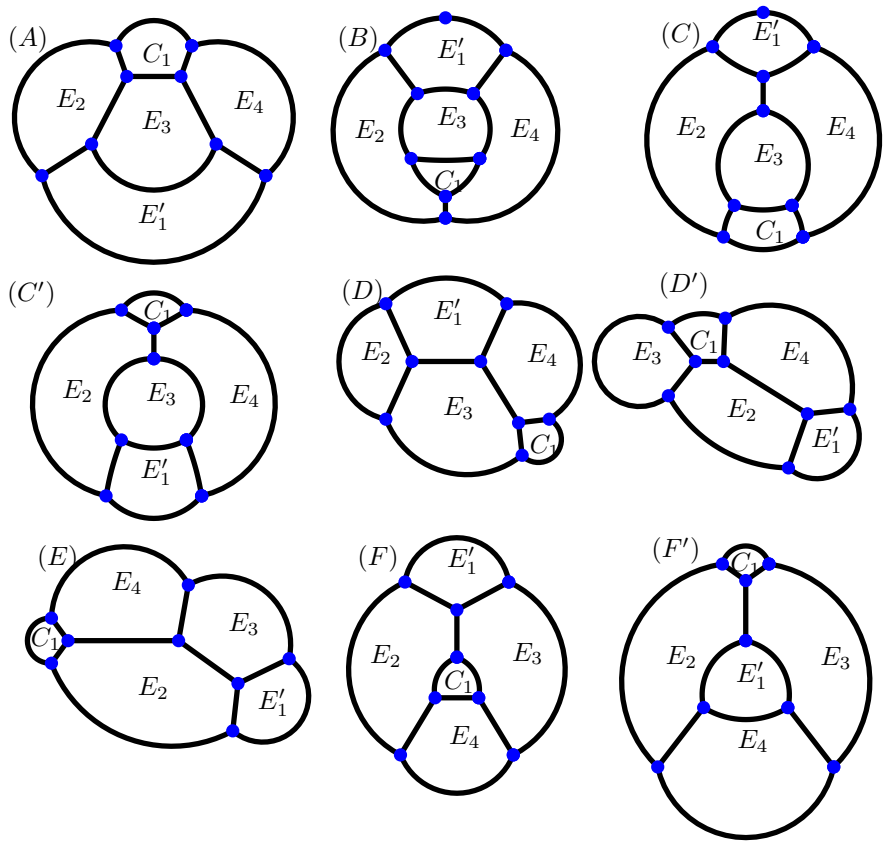

\begin{proposition}\label{classify_five_components}
Let $\E\in \M^*(1,1,1,1)$ be a cluster with 5 connected components. Then, up to
a relabeling of the components, the
topology of $\E$ is one of the cases represented in
Figure~\ref{classificazione}.
\end{proposition}

\begin{proof}
Suppose that $E_1$ is the only disconnected region and let $E_1'$ and
$C_1$ respectively be the big and small connected components of $E_1$.
By Proposition~\ref{prop:edges} we know that $\partial\E$ is composed by 12
edges and 8 vertices moreover both $E_1'$ and $C_1$ may have at most 4
edges if they are external and 3 edges if they are internal.

\emph{Step 1.} Suppose that both $E_1'$ and $C_1$ have four
edges (and hence they are external). All the 8 vertices of the cluster are
vertices of either $E_1'$ or $C_1'$ and both $E_1'$ and $C_1$ have an
external edge with two external vertices. The external region $E_0$
has four edges.

The remaining two internal vertices
of $E_1'$ must be connected with the two internal vertices of
$C_1$ (remember that we cannot have two edges with the same end
points, because two-sided components are not allowed).
Hence the cluster is of type (A) in
Figure~\ref{classificazione}.

\emph{Step 2.} Suppose that $E_1'$ has 4 edges (hence it is external)
and suppose $C_1$ is external with 3 edges. In this case we need to
add an additional vertex $v$.

If $v$ is external then the external region $E_0$ has five edges.
The vertex $v$ must be connected to an internal vertex of
$E_1'$ while the other internal vertex of $E_1'$ must be connected to
the internal vertex of $C_1$. The resulting topology is (D).

If, instead, the additional vertex $v$ is internal, it must be
connected to the two internal vertices of $E_1'$ and to the internal
vertex of $C_1$. Hence we are in case (C').

\emph{Step 3.} Suppose $E_1'$ has 4 edges (hence it is external) and
suppose $C_1$ is internal with 3 edges. Since the external region must
have at least three edges, there is an additional external vertex $v$
and $E_0$ has three edges.
One of the
three vertices of $C_1$ must be connected to the vertex $v$ while the
other two vertices of $C_1$ must be connected to the two internal
vertices of $E_1'$. The resulting topology is (B).

\emph{Step 4.} Suppose $E_1'$ has 3 edges and is external while $C_1$
has four edges (and hence is external). We repeat the same reasoning
of Step 2 with $E_1'$ and $C_1$ exchanged and we obtain cases (D') and
(C).

\emph{Step 5.} Suppose $E_1'$ has 3 edges and is internal while $C_1$
has four edges (and hence is external). We repeat the same reasoning
of Step 3 and obtain case (B) with $E_1'$ and $C_1$ exchanged. But in
this case we would have two big internal components: $E_1'$ and $E_3$
and this is impossible in view of Corollary~\ref{cor:one_big_internal}.

\emph{Step 6.} Suppose that both $E_1'$ and $C_1$ have three edges and
are external. There are two additional vertices $v,w$
which are not vertices of $E_1'$ or $C_1$.
Since the external region
$E_0$ has at most 5 edges (there are only 5 bounded components) one of
the two vertices, say $v$, is internal. The other vertex $w$ cannot be
internal, because otherwise $v$ and $w$ need to be joined by two
different edges, which is not possible. The internal vertex $v$ must
be connected to $w$ and to the two internal vertices of $E_1'$ and
$C_1$. The resulting topology is (E).

\emph{Step 7.} Suppose that both $E_1'$ and $C_1$ have three edges and
suppose $E_1'$ is external and $C_1$ is internal. We need to place two
additional vertices $v$ and $w$. Certainly one among $v$ and $w$ is external,
since $E_0$ has at least three edges. In case
both $v$ and $w$ are external $E_0$ has four edges.

If two of the three vertices of $C_1$ are connected to the same
vertex, we would obtain an additional three sided component (say it is
$E_2$).
Hence we notice we have three components with three edges: $E_1'$,
$C_1$ and $E_2$. 
 Let $n_0$, $n_3$ and
$n_4$ be the number of edges of $E_0$, $E_3$ and $E_4$. By Euler's formula we
have
$24=3 \times 3 +n_0+ n_3 + n_4 \leq 9+4+n_3+n_4$,
which means that $\max\{n_3,n_4\}\ge \frac{11}{2}$, i.e $\max\{n_3,n_4\}\ge 6$
(notice that $n_3$ and $n_4$ are integers), which is impossible
by Proposition~\ref{prop:edges} (each component can only have one edge
in common with each other component).

So the three vertices of $C_1$ are connected to $v$, $w$ and to the
internal vertex of $E_1'$. Necessarily $v$ and $w$ are also connected
to the external vertices of $E_1'$ hence they are both external and
$E_0$ has 4 edges. The resulting cluster is of type (F).

\emph{Step 8.} Suppose that both $E_1'$ and $C_1$ have three edges and
suppose that $E_1'$ is internal and $C_1$ is external. We obtain the
same classification of Step 7 but with $E_1'$ and $C_1$
exchanged. We obtain case (F').

\emph{Step 9.} Suppose that both $E_1'$ and $C_1'$ have three edges
and are both internal. This is impossible because the external region
would only have two edges, which is excluded.
\end{proof}

\begin{proposition}\label{BCCDDEF}
  Let $\E\in\M^*(1,1,1,1)$. Then $\E$ cannot have the topologies $(B)$,
  $(C)$, $(C')$, $(D)$, $(D')$, $(E)$, $(F)$ of Figure~\ref{classificazione}.
\end{proposition}

\begin{proof}
  Notice that in each case it is possible (by subsequently removing
  triangular components) to reduce the cluster $\E$
  to a double bubble $(E_1'', E_2'')$ where $E_1''\supset E_1'$
  and $E_2''\supset E_2$.

  So, by applying Lemma~\ref{lemma:double_bubble_pressure} we obtain
  at once
  \begin{equation}\label{eq:p1_upper}
  p_1 \le \frac{k_8}{\sqrt{\min\{m(E_1'),m(E_2)\} }}
  = \frac{k_8}{\sqrt{1-m(C_1)}}
  \le \frac{k_8}{\sqrt{1-k_1}} \le 1.7352.
  \end{equation}

  In the case when $C_1$ has only three edges (i.e.\ cases $(B)$, $(C')$,
  $(D)$, $(E)$, and $(F)$) we can apply
  Proposition~\ref{prop:pressure_perimeter} and then
  Proposition~\ref{prop:k_1} to obtain
  \[
  p_1 \ge \frac{(6-3)\pi}{3P(C_1)}
  \ge \frac{\pi}{k_7}
  \ge 2.2125
  \]
  and this is in contradiction with \eqref{eq:p1_upper}.

  In both cases $(C)$ and $(D')$ we can reduce the triangular components
  to find a double bubble $(E_2'', E_4'')$ with $E_2''\supset E_2$ and
  $E_4'' \supset E_4$. Moreover $E_2''\subset \RR^2\setminus (E_0\cup
  E_4)$ and $E_4'' \subset \RR^2\setminus(E_0\cup E_2)$ so that
  $m(E_2'')\le 3$ and $m(E_4'')\le 3$. So, by using
  Lemma~\ref{lemma:double_bubble_pressure} we obtain
  \[
  \min\{p_2,p_4\} \ge \frac{k_8}{\sqrt{\min\{4 - m(E_4),4 - m(E_2)\}}}
  = \frac{k_8}{\sqrt 3}.
  \]
  In case $(D')$ we can find another reduction to a double bubble
  $(E_2'', E_3'')$ and, as before, we find
  \[
  \min\{p_2,p_3\} \ge \frac{k_8}{\sqrt 3}
  \]
  so that, in this case, $\min\{p_2,p_3,p_4\}\ge k_8/\sqrt{3}$.

  In case $(C)$ we apply Proposition~\ref{prop:pressure_perimeter} to
  the component $C_1$ to obtain:
  \[
  p_1 \ge \frac{(6-4)\pi}{3P(C_1)} \ge \frac{2\pi}{3k_7}\ge 1.4750
  > 0.9179 \ge \frac{k_8}{\sqrt 3}
  \]
  and then we apply the same Proposition~\ref{prop:pressure_perimeter} to $E_3$ to obtain (notice
  that we consider $\ell=0$ since $E_3$ is internal):
  \[
  p_3 \ge \frac{(6-3)\pi}{3P(E_3)} + \min\{p_1,p_2,p_4\}
  \ge\min\{p_1,p_2,p_4\} \ge \frac{k_8}{\sqrt 3}.
  \]

  So, in both cases $C$ and $D'$, we obtain
  \[
  \min\{p_2,p_3,p_4\} \ge \frac{k_8}{\sqrt 3}.
  \]
  Now we need to estimate the length $\ell$ of the external edge of
  $C_1$. By Proposition~\ref{prop:varII} we have (notice that $m(C_1)<k_1<1$),
  \[
  \ell \le \frac{m(C_1)\cdot P(\E)}{\lvert 2 -m(C_1)\rvert}
  \le \frac{m(C_1)k_0}{2-k_1}
  \]
  while, by Proposition~\ref{prop:k_1}, we have
  \[
  P(C_1) \le k_7.
  \]
  By applying Proposition~\ref{prop:pressure_perimeter}, and
  using the previous estimates, we get
  \begin{align*}
  p_1 & \ge \frac{(6-4)\pi}{3P(C_1)} + \min\{p_2,p_3,p_4\}
  \left(1-\ell\cdot\frac{1}{P(C_1)}\right) \\
  &\ge \frac{2\pi}{3 k_7} + \frac{k_8}{\sqrt 3}\left(1-\frac{m(C_1)k_0}{2
    - k_1}\cdot\frac{1}{2\sqrt{\pi}\sqrt{m(C_1)}}\right) \\
  &\ge \frac{2\pi}{3 k_7} + \frac{k_8}{\sqrt 3}\left(1-\frac{\sqrt{k_1}
    k_0}{2\sqrt \pi(2-k_1)}\right) \ge 1.7615
  \end{align*}
  which, again, is in contradiction with~\eqref{eq:p1_upper}.
\end{proof}

\begin{proposition}\label{prop:F'}
  Let $\E\in \M^*(1,1,1,1)$ has 5 components. Then we exclude that
  $\E$ has the topology $(F')$ of Figure~\ref{classificazione}.
\end{proposition}
\begin{proof}
  By removing the triangular components we are able to reduce the
  cluster $\E$ to a double bubble $(E_2'',E_3'')$ with $E_2''\supset
  E_2$ and $E_3''\supset E_3$.
  Notice that $E_2''\subset \RR^2
  \setminus(E_0 \cup E_3)$ and $E_3''\subset \RR^2\setminus(E_0 \cup
  E_2)$ so that $m(E_2'')\le 3$ and $m(E_3'')\le 3$. So, by
  Lemma~\ref{lemma:double_bubble_pressure}, we obtain
  \[
  \min\{p_2,p_3\} \ge \frac{k_8}{\sqrt 3}.
  \]

  We repeat the same argument with $E_4$ in place of $E_3$ to obtain
  $\min\{p_2,p_4\} \ge \frac{k_8}{\sqrt 3}$ so that
  \[
  \min\{p_2,p_3,p_4\} \ge \frac{k_8}{\sqrt 3}.
  \]

  Now we estimate the length $\ell$ of the external edge of
  $C_1$ by using Proposition~\ref{prop:varII}:
  \[
  \ell \le \frac{m(C_1)}{\lvert 2-m(C_1)\rvert}\cdot P(\E)
  \]
  i.e. (notice that $m(C_1)<k_1<1$)\
  \begin{align*}
  \frac{\ell}{P(C_1)}  \le \frac{\ell}{2\sqrt{\pi}\sqrt{m(C_1)}}
   \le \frac{\sqrt{m(C_1)}P(\E)}{2\sqrt{\pi}(2-m(C_1))}
   \le \frac{\sqrt{k_1}k_0}{2\sqrt{\pi}(2-k_1)}
  \end{align*}
  and we apply Proposition~\ref{prop:pressure_perimeter} to obtain
  \begin{align*}
  p_1 & \ge \frac{(6-3)\pi}{3P(C_1)} + \min\{p_2,p_3\} \left(1-
  \frac{\ell}{P(C_1)}\right)\\
  & \ge \frac{\pi}{k_7} + \frac{k_8}{\sqrt
    3}\left(1-\frac{\sqrt{k_1}k_0}{(2-k_1)2\sqrt{\pi}}\right) \ge c_8 \defeq 2.4990.
  \end{align*}

  By Lemma~\ref{lem:turning_angle} applied to the component
  $E_1'$ we have
  \begin{align*}
  \pi = \frac{(6-3)\pi}{3} = \sum_{j=0}^4 (p_1-p_j)L_j &\ge (p_1 -
  \max\{p_0,p_2,p_3,p_4\}) P(E_1')\\
  & = (p_1 - \max\{p_2,p_3,p_4\}) 2 \sqrt{\pi}\sqrt{1-k_1}
  \end{align*}
  so that
  \[
  \max\{p_2,p_3,p_4\} \ge p_1 - \frac{\sqrt\pi}{2\sqrt{1-k_1}}
  \]
  Hence
  \begin{align*}
  P(\E) &= 2 (p_1+p_2+p_3+p_4) \ge 2 (p_1 + \max\{p_2,p_3,p_4\} + 2 \min\{ p_2,p_3,p_4\})\\
  &\ge 4 c_8  - 2\cdot\frac{\sqrt\pi}{2\sqrt{1-k_1}} + 4\cdot
  \frac{k_8}{\sqrt 3}
  \ge 11.5561 \ge k_0 
  \end{align*}
  which is a contradiction.
\end{proof}

\begin{proposition}\label{excludeA}
Let $\E\in \M^*(1,1,1,1)$ be a cluster with 5 components. Then we
exclude that $\E$ has the topology $(A)$ depicted in Figure~\ref{classificazione}.
\end{proposition}
\begin{proof}
First of all notice that
\begin{align*}
2k_0 \ge  2P(\E) &= P(E_1') + P(C_1) + P(E_2) + P(E_4) + P(E_0)
 + P(E_3)\\
   &\ge 2 \sqrt \pi \left(\sqrt{1-k_1} + \sqrt{k_2} + 2\sqrt{1} +
   \sqrt{4}\right) + P(E_3)
\end{align*}
so that
\[
P(E_3) \le 2 k_0 - 2\sqrt \pi \left(\sqrt{1-k_1} + \sqrt{k_2} +
4\right) \le c_9 \defeq 4.4111.
\]

Now let $\ell_j$ be the total length of the external edges of the region
$E_j$ ($j=1,2,4$). If we remove $E_1$ from $\E$ we obtain a 3-cluster
$\E'=(E_2,E_3,E_4)$ with $\E'\in\mathcal C^*(1,1,1)$. Hence, by
Lemma~\ref{lemma:ptb} we have $P(\E') \ge k_{10}$. Moreover
\[
  \ell_1 = P(\E) - P(\E') \le k_0 - k_{10}.
\]
We can repeat the same argument for $\ell_2$ and $\ell_4$ to obtain
\begin{equation}\label{eq:ell124}
  \max\{\ell_1,\ell_2,\ell_4\} \le k_0 - k_{10}.
\end{equation}

By Proposition~\ref{prop:pressure_perimeter} we have (notice that we
let $\ell=0$ since $E_3$ is internal)
\begin{equation}\label{eq:p124}
p_3 \ge \frac{(6-4)\pi}{3P(E_3)} + \min\{p_1,p_2,p_4\}
> \min\{p_1,p_2,p_4\}.
\end{equation}
The same proposition applied to the component $C_1$ gives
\[
p_1 \ge \frac{(6-4)\pi}{3P(C_1)} \ge \frac{2\pi}{3k_7} \ge 1.4750.
\]
Since
\begin{align*}
k_0 &\ge P(\E) = 2(p_1+p_2+p_3+p_4) \ge 2 p_1 + 6 \min\{p_2,p_3,p_4\} \\
&\ge \frac{4 \pi}{3 k_7} + 6 \min\{p_2,p_3,p_4\},
\end{align*}
we obtain
\[
\min\{p_2,p_3,p_4\} \le \frac{k_0}{6} - \frac{2\pi}{9k_7} \le 1.3744,
\]
so that
\begin{equation}\label{eq:p234}
p_1 > \min\{p_2,p_3,p_4\}.
\end{equation}

Putting together \eqref{eq:p124} and \eqref{eq:p234} we can say that
the minimum among $p_1$, $p_2$, $p_3$, $p_4$ is either $p_2$ or $p_4$. Without
loss of generality we can assume that such a minimum is $p_2$.

Hence, applying Lemma~\ref{lem:turning_angle} to the region
$E_2$ we obtain
\[
\frac{(6-4)\pi}{3} = \sum_{i=0}^4 (p_2-p_i) L_i \le p_2 \ell_2
\]
where $L_i$ is the total length of the edges between $E_2$ and $E_i$
(so that $L_0=\ell_2$) and we used the estimate $p_2-p_i\le 0$ for
$i\neq 0$.
So, using~\eqref{eq:ell124}
\[
\min\{p_1,p_2,p_3,p_4\} = p_2 \ge \frac{2\pi}{3\ell_2} \ge \frac{2\pi}{3(k_0-k_{10})}.
\]

Now, use again Proposition~\ref{prop:pressure_perimeter} on the region $E_3$
to obtain
\begin{equation}\label{eq:p3bis}
 p_3 \ge \frac{(6-4)\pi}{3P(E_3)} + \min\{p_1,p_2,p_4\}
\ge \frac{2\pi}{3 c_9} + \frac{2\pi}{3(k_0-k_{10})}
\ge c_{10} \defeq 1.3466.
\end{equation}

Finally we apply Lemma~\ref{lem:turning_angle} to the region
$E_0$ to obtain
\[
  \frac{(6+4)\pi}{3} = p_1 \ell_1 + p_2 \ell_2 + p_4 \ell_4
  \le \max\{\ell_1,\ell_2,\ell_4\} (p_1+p_2+p_4)
\]
hence, using also~\eqref{eq:ell124}
\[
p_1 + p_2 + p_4 \ge \frac{10\pi}{3(k_0-k_{10})}.
\]
So, using also~\eqref{eq:p3bis}, we have
\[
 P(\E) = 2 p_3 + 2(p_1+p_2+p_4) \ge 2 c_{10}+
 \frac{20\pi}{3(k_0-k_{10})}
 \ge 11.4116 > k_0
\]
which is a contradiction.
\end{proof}



\begin{theorem}\label{thm:connected}
Let $\E\in \M(1,1,1,1)$. Then $\E$ is connected.
\end{theorem}
\begin{proof}
 By Corollary~\ref{corollary:strong_weak} we know that $\E\in\M^*(1,1,1,1)$.

 By Corollary~\ref{cor:one_big_component} and by Proposition~\ref{prop:two_small}
 we know that each region
 $E_i$ has exactly one big component and the total number of
 small components is not larger than two.

 If the cluster has exactly
 two small components, with
 Corollary~\ref{cor:not_three}, we exclude that they belong to the same
 region and with Proposition~\ref{prop:not_two_two} we
 exclude that they belong to two different regions.

 Finally, from Proposition~\ref{classify_five_components}
 and Propositions~\ref{BCCDDEF}, \ref{prop:F'} and~\ref{excludeA}
 we exclude that the cluster has exactly one small connected component
 (which means five connected components in total).
\end{proof}

\section{Connected clusters (four components)}\label{sec:four}

\begin{proposition}\label{prop:two_topologies}
  Let $\E\in \M^*(1,1,1,1)$ be a connected cluster.
  Then $\E$ has one of the two topologies depicted in Figure~\ref{sandwich_flower}.
\end{proposition}
\begin{proof}
  Since every region is connected, by Proposition~\ref{prop:edges}
  every region (comprising $E_0$) has three or four edges and the
  cluster has a total of nine edges and six vertices. Let $x$ be the number of regions
  (bounded or unbounded) with four edges and let $y$ the number of
  regions (bounded or unbounded) with three
  edges. We have one unbounded region $E_0$ and four bounded regions,
  hence: $x+y=5$. Moreover summing up all the edges of all the regions
  we would count each edge twice, hence we have: $4x+3y=18$. Solving
  the system of two equations gives $x=3$, $y=2$ hence we have three
  regions with four edges and two regions with four edges.

  If the unbounded region $E_0$ has three edges (note that there is a total of six vertices), there is one internal
  region and three external regions. The internal region can only have
  three edges (because it is not adjacent to $E_0$) and we are in the
  first case of the statement.

  If the unbounded region $E_0$ has four edges, all the bounded
  regions are external: two of them have three edges and two have four
  edges. The regions with four edges are adjacent to all other
  regions hence the regions with three edges don't touch each
  other. We are in the second case of the statement.
\end{proof}

\begin{proposition}\label{prop:no_flower}
Let $\E \in \M^*(1,1,1,1)$ be a connected cluster. Then $\E$ has not
the flower topology.
\end{proposition}

\begin{proof}
Suppose by contradiction that $\E$ has the flower topology and
let $E_1$ be the internal three sided region.

First of all notice that
\begin{align*}
  k_0 \ge  P(\E) \geq P(E_0)+ P(E_1)
  \ge 2 \sqrt \pi \sqrt{4} + P(E_1)
\end{align*}
so that
\[
P(E_1) \le k_0 - 4\sqrt \pi  \le c_{11} \defeq 4.1064.
\]

Now let $\ell_2$ be the length of the external edge of the region
$E_2$. If we remove $E_2$ from $\E$ we obtain a 3-cluster
$\E'=(E_1,E_3,E_4)$ with $\E'\in\mathcal C^*(1,1,1)$. Hence, by
Lemma~\ref{lemma:ptb} we have $P(\E') \ge k_{10}$. Moreover
\[
  \ell_2 = P(\E) - P(\E') \le k_0 - k_{10}.
\]
We can repeat the same argument for the lengths $\ell_3$ and $\ell_4$
of the external edges of $E_3$ and $E_4$, to obtain
\begin{equation}\label{eq:ell234}
  \max\{\ell_2,\ell_3,\ell_4\} \le k_0 - k_{10}.
\end{equation}


By removing the triangular components we are able to reduce the
  cluster $\E$ to a double bubble $(E_2'',E_3'')$ with $E_2''\supset
  E_2$ and $E_3''\supset E_3$.
  Notice that $E_2''\subset \RR^2
  \setminus(E_0 \cup E_3)$ and $E_3''\subset \RR^2\setminus(E_0 \cup
  E_2)$ so that $m(E_2'')\le 3$ and $m(E_3'')\le 3$. So, by
  Lemma~\ref{lemma:double_bubble_pressure}, we obtain
  \[
  \min\{p_2,p_3\} \ge \frac{k_8}{\sqrt 3}.
  \]

  We repeat the same argument with $E_4$ in place of $E_3$ to obtain
  $\min\{p_2,p_4\} \ge \frac{k_8}{\sqrt 3}$ so that
  \[
  \min\{p_2,p_3,p_4\} \ge \frac{k_8}{\sqrt 3}.
  \]


Now, use again Proposition~\ref{prop:pressure_perimeter} on the region $E_1$
to obtain (notice that we let $\ell =0$ since $E_1$ is internal)
\begin{equation}\label{eq:p1}
 p_1 \ge \frac{(6-3)\pi}{3P(E_1)} + \min\{p_2,p_3,p_4\}
\ge \frac{\pi}{c_{11}} + \frac{k_8}{\sqrt 3}
\ge c_{12} \defeq 1.6829.
\end{equation}

Finally we apply Lemma~\ref{lem:turning_angle} to the region
$E_0$ to obtain
\[
  \frac{(6+3)\pi}{3} = p_2 \ell_2 + p_3 \ell_3 + p_4 \ell_4
  \le \max\{\ell_1,\ell_2,\ell_4\}\cdot (p_2+p_3+p_4)
\]
hence, using also~\eqref{eq:ell234}
\[
p_2 + p_3 + p_4 \ge \frac{3\pi}{k_0-k_{10}}.
\]
So, using also~\eqref{eq:p1}, we have
\[
 P(\E) = 2 p_1 + 2(p_2+p_3+p_4) \ge 2 c_{12} +
 \frac{6\pi}{k_0-k_{10}}
 \ge 11.2124 > k_0
\]
which is a contradiction.
            \end{proof}

\begin{theorem}\label{thm:sandwich}
  Let $\E\in \M(1,1,1,1)$. Then $\E$ has the \emph{sandwich} topology
  as in Figure~\ref{sandwich_flower}.
\end{theorem}
\begin{proof}
  By Theorem~\ref{thm:connected} we know that $\E$ is connected so by
  Proposition~\ref{prop:two_topologies} we know that $\E$ can either
  have the flower or the sandwich topology. With
  Proposition~\ref{prop:no_flower} we exclude the flower topology
  and the result follows.
\end{proof}

\begin{conjecture}\label{conj:sandwich}
  Up to isometries there is a unique stationary cluster $\E\in \mathcal C(1,1,1,1)$ with the
  \emph{sandwich} topology. In such a cluster the regions $E_1$ and
  $E_3$ are isometric to $E_2$ and $E_4$ (respectively).
\end{conjecture}

  A stationary sandwich cluster with given external radii $r_1, r_2, r_3, r_4$
  can be uniquely constructed by taking the double bubble with
  external radii $r_1, r_2$ and then ``growing'' a triangular region
  in each vertex so that it reaches the prescribed external radius
  In fact both the area and the radii of a growing triangular camera
  are strictly increasing, see \cite{MM, W}
  
  Not only that, since when the triangular region grows, the area of
  the quadrangular regions decreases, if $r_1$ and $r_2$ are fixed it
  is possible to find $r_3=r_4$ so that the area of the triangular
  regions become equal to the area of the smaller quadrangular region.
  If $r_1=r_2$ we then find that all four regions have equal area and
  with a rescaling we obtain a stationary symmetric
  sandwich cluster in $\mathcal C(1,1,1,1)$.

  We believe that when $r_1>r_2$ then the previous construction
  would yield a cluster where
  $m(E_1) > m(E_2) = m(E_3) = m(E_4)$. This would exclude that asymmetric sandwich cluster can
  have all equal areas.

  We have some numerical computations \cite{github} which
  confirm this.

\end{document}